\documentclass{amsart}
\usepackage{ae,aecompl}
\usepackage{chngcntr}
\usepackage{graphicx}
\usepackage[all,cmtip]{xy}
\usepackage{amsmath, amscd}
\usepackage{amsthm}
\usepackage{amssymb}
\usepackage{amsfonts}
\usepackage{qsymbols}
\usepackage{latexsym}
\usepackage{cite}
\usepackage{color}
\usepackage{url}
\usepackage{verbatim}

\usepackage{microtype}

\newcommand {\abs}[1]{\lvert#1\rvert}

\newcommand {\C}{{\mathbb C}}

\newcommand {\Ce}{\mathrm{C}}
\newcommand {\Cc}{\mathrm{C_{c}}}
\newcommand {\Cb}{\mathrm{C_{b}}}

\newcommand {\Cstar}{{\mathrm C}^\ast}

\newcommand {\Di}{{\mathbb D}}

\newcommand {\ERG}{{\mathcal{E}}}
\newcommand {\Ell}{\mathrm{L}}
\newcommand {\Ellp}{\Ell^{p}}

\newcommand {\INV}{{\mathcal{I}}}

\newcommand{\ind}{{\mathbf{1}}}

\newcommand {\La}{{\mathcal{L}}}

\newcommand {\M}{{\mathrm{M}}}

\newcommand {\N}{{{\mathbb N}}}
\newcommand {\norm}[1]{\left\|#1\right\|}
\newcommand {\ph}{{\varphi}}
\newcommand {\R}{{\mathbb R}}

\newcommand {\supp}{{\mathrm{supp}}}
\newcommand {\T}{{\mathbb T}}
\newcommand {\ud}{\mathrm{d}}
\newcommand {\ue}{\mathrm{e}}
\newcommand {\ui}{\mathrm{i}}
\newcommand {\vanish}[1]{\relax}
\newcommand {\w}{{\omega}}

\newcommand {\bs}{{{B}}}

\newcommand {\isom}{{{S}}}

\newcommand {\measalg}{{\mathcal A}_\mu}
\newcommand {\rep}{\rho}
\newcommand{\repint}{\rho}
\newcommand {\ts}{{{X}}}
\newcommand {\sect}{{{s}}}

\newcommand {\vs}{{{V}}}
\newcommand {\PR}{{\mathcal{P}}}

\newcommand {\ks}{{\Ell^p}}
\newcommand {\kstext}{{\Ell^p}}
\newcommand {\two}{{\Ell^2}}
\newcommand {\twotext}{{\Ell^2}}

\newcommand{\lh}{T}

\newtheorem{theorem}{Theorem}[section]
\newtheorem{lemma}[theorem]{Lemma}
\newtheorem{proposition}[theorem]{Proposition}
\newtheorem{corollary}[theorem]{Corollary}

\theoremstyle{definition}
\newtheorem{definition}[theorem]{Definition}
\newtheorem{remark}[theorem]{Remark}

\numberwithin{equation}{section}

\title[Disintegration of positive isometric group representations]{Disintegration of positive isometric group representations on $\boldsymbol{\Ellp}$-spaces}

\author{Marcel de Jeu}
\address{Mathematical Institute\\
Leiden University\\
P.O.~Box 9512\\
2300 RA Leiden\\
The Netherlands}
\email{mdejeu@math.leidenuniv.nl}

\author{Jan Rozendaal}
\address{Institute of Mathematics of the Polish Academy of Sciences\\
ul. \'{S}niadeckich 8\\
00-656 Warsaw\\
Poland}
\email{janrozendaalmath@gmail.com}

\thanks{During the preparation of this manuscript the second author was supported by NWO-grant 613.000.908.}


\dedicatory{}


\begin{document}

\begin{abstract}
Let $G$ be a Polish locally compact group acting on a Polish space~$\ts$ with a $G$-invariant probability measure~$\mu$. We factorize the integral with respect to $\mu$ in terms of the integrals with respect to the ergodic measures on $X$, and show that $\Ell^{p}(\ts,\mu)$ ($1\leq p<\infty$) is $G$-equivariantly isometrically lattice isomorphic to an $\kstext$-direct integral of the spaces $\Ell^{p}(\ts,\lambda)$, where $\lambda$ ranges over the ergodic measures on $X$. This yields a disintegration of the canonical representation of $G$ as isometric lattice automorphisms of $\Ell^{p}(\ts,\mu)$ as an $\kstext$-direct integral of order indecomposable representations.

If $(\ts^\prime,\mu^\prime)$ is a probability space, and, for some $1\leq q<\infty$, $G$ acts in a strongly continuous manner on $\Ell^{q}(\ts^\prime,\mu^\prime)$ as isometric lattice automorphisms that leave the constants fixed, then $G$ acts on $\Ell^{p}(\ts^{\prime},\mu^{\prime})$ in a similar fashion for all $1\leq p<\infty$. Moreover, there exists an alternative model in which these representations originate from a continuous action of $G$ on a compact Hausdorff space. If $(\ts^\prime,\mu^\prime)$ is separable, the representation of $G$ on $\Ell^p(X^\prime,\mu^\prime)$ can then be disintegrated into order indecomposable representations.

The notions of $\kstext$-direct integrals of Banach spaces and representations that are developed extend those in the literature.




\end{abstract}

\maketitle

\section{Introduction and overview}
\label{sec:introduction}

There is an extensive literature on unitary group representations. Apart from an intrinsic interest and mathematical relevance, the wish to understand such representations originates from quantum theory, where the unitary representations of the symmetry group of a physical system have a natural role. However, in many cases where a symmetry yields a unitary representation of the pertinent symmetry group, there is also a family of canonical representations on Banach lattices. The rotation group of $\R^3$ acts on the 2-sphere in a measure preserving fashion, yielding a canonical unitary representation on $\Ell^2(S^2,\ud\sigma)$, but there are, in fact, canonical strongly continuous representations as isometric lattice automorphisms of the (real) Banach lattice $\Ell^{p}(S^2,\ud\sigma)$ for all $1\leq p<\infty$. Likewise, for all $1\leq p<\infty$, the motion group of $\R^{d}$ acts in a strongly continuous fashion as isometric lattice automorphisms on the Banach lattice $\Ell^{p}(\R^{d},\ud x)$. Representations of groups as isometric lattice automorphisms of Banach lattices are quite common. In spite of this, not much is known about such representations or, for that matter, about the related positive representations of ordered Banach algebras and Banach lattice algebras in Banach lattices; the material in \cite{deJeu-Messerschmidt13, deJeu-Ruoff16, deJeu-Wortel12, deJeu-Wortel14, Wickstead15a, Wickstead15b} is a modest start at best. Nevertheless, it seems quite natural to investigate such representations. Moreover, given the long-term success, in a Hilbert space context, of the passage from single operator theory to groups and algebras and their representations\textemdash a development that was initially also stimulated and guided by the wish to understand unitary group representations\textemdash it seems promising to develop a similar theory for representations in Banach lattices.

One of the highlights in abstract representation theory in Hilbert spaces is the insight that every strongly continuous unitary representation of a separable locally compact Hausdorff group on a separable Hilbert space can be disintegrated into irreducible unitary representations. This follows from a similar theorem for $\Cstar$-algebras and the standard relation between the unitary representations of a group and the non-degenerate representations of its group $\Cstar$-algebra; see \cite[Theorem~8.5.2 and 18.7.6]{Dixmier77}. Every representation is thus built from irreducible ones. Is something analogous possible for strongly continuous actions of a locally compact Hausdorff group as isometric lattice automorphisms of Banach lattices? This seems a natural guiding question when studying representations in an ordered context. It is still very far from having been answered in general, and presumably one will have to restrict oneself to a class of suitable Banach lattices. After all, the unitary theory works particularly well in just one space, namely $\ell^2$, and it seems doubtful that there can be a uniform answer for the existing diversity of Banach lattices.

What, exactly, should `irreducible' mean in an ordered context? When searching for the parallel with unitary representations it is actually more convenient to think of irreducible unitary representations as indecomposable unitary representations, which happens to be the same notion, and look for the analogue of the latter. Given a representation of a group $G$ as lattice automorphisms of two vector lattices $E_{1}$ and $E_{2}$, there is a natural representation of $G$ as lattice automorphisms of the vector lattice $E=E_{1}\oplus E_{2}$. If a representation of $G$ as lattice automorphisms of a given vector lattice $E$ is not such an order direct sum of two non-trivial sub-representations, then one will want to call it order indecomposable. Actually, if $E=E_{1}\oplus E_{2}$ is an order direct sum of vector lattices, then more is true than one would perhaps expect. $E_{1}$ and $E_{2}$ are automatically projection bands, and they are each other's disjoint complement; this is a special case of \cite[Theorem~11.3]{Zaanen}. Coming from the other side, if a projection band in $E$ is invariant under a group of lattice automorphisms, then so is its disjoint complement, and hence there is a corresponding decomposition of the representation into two sub-representations as lattice automorphisms. All in all, we have the following natural definition.

\begin{definition}\label{def:indecomposable}
Let $E$ be a vector lattice, and let $\rep$ be a homomorphism from $G$ into the group of lattice automorphisms of $E$. Then the representation $\rep$ is \emph{order indecomposable} if $\{0\}$ and $E$ are the only $G$-invariant projection bands in $E$.
\end{definition}

Note that $G$ acts on $E$ as lattice automorphisms precisely when it acts as positive operators; hence one can also refer to such a representation as a \emph{positive representation} of $G$ on $E$.

It is a non-trivial fact that an order indecomposable positive representation of a finite group on a Dedekind complete vector lattice is finite dimensional; this follows from \cite[Theorem~3.14]{deJeu-Wortel12}. It is also possible to show that every finite dimensional positive representation of a finite group on an Archimedean vector lattice is an order direct sum of order indecomposable positive representations, where the latter can be classified \cite[Theorem~4.10 and Corollary~4.11]{deJeu-Wortel12}. This answers our question about disintegrating finite dimensional positive representations of finite groups.  The matter is still open for infinite dimensional positive representations of finite groups.

For positive representations of an abstract group $G$ on a normalized Banach sequence space $E$, it is true that the representation is a (generally infinite) order direct sum of order indecomposable positive representations; see \cite[Theorem~ 5.7]{deJeu-Wortel14}. If the group has compact image in the strong operator topology, and $E$ has order continuous norm (this includes the spaces $\ell^{p}$ for $1\leq p<\infty$), then these order indecomposable positive representations are all finite dimensional. This is an analogue of the well known theorem for unitary representations of compact Hausdorff groups.

At the time of writing, not much (if anything) seems to be known about disintegration of positive group representations into order indecomposable representations beyond the above results. These are both concerned with compact groups and, analogously to the unitary case, the disintegration is then a discrete summation. In the present paper, a technically more challenging context is considered, and we consider a class of positive representations where the disintegration can be of a truly continuous nature. This disintegration is obtained in two main steps.

This first main step\textemdash we omit the necessary conditions for the sake of clarity\textemdash consists of a disintegration into order indecomposable representations of the representations of a locally compact Hausdorff group $G$ as isometric lattice automorphisms of $\Ell^{p}$-spaces, as canonically associated with an action of $G$ on a Borel probability space $(\ts,\mu)$ with invariant measure $\mu$. Such a representation is order indecomposable precisely when $\mu$ is ergodic. One might therefore hope that, somehow, a disintegration of $\mu$ into ergodic measures $\lambda$ will yield a disintegration of the canonical positive representation on $\Ell^{p}(\ts,\mu)$ in terms of the order indecomposable canonical representations on $\Ell^{p}(\ts,\lambda)$ for ergodic $\lambda$. This can in fact be done, and Theorem~\ref{thm:disintegrating_space_actions} clarifies what `somehow' is here: in a $G$-equivariant fashion, the Banach lattice $\Ell^p(\ts,\mu)$ is an $\kstext$-direct integral of the Banach lattices $\Ell^p(\ts,\lambda)$ for ergodic $\lambda$, where the $\kstext$-direct integral is with respect to a Borel probability measure on the set of ergodic measures. Apart from the framework of direct integrals of Banach spaces as such, which could also have representation theoretical applications in other contexts, the principal ingredient for the proof of this result is a factorization of the integral over $\ts$ with respect to $\mu$ in terms of those with respect to the ergodic measures; see Theorem~\ref{thm:factorization}. In spite of its aesthetic appeal, we are not aware of a reference for the pertinent formula in this Tonelli--Fubini-type theorem, which itself is based on the aforementioned disintegration of $\mu$ into ergodic measures.

Aside, let us briefly mention that there is no uniqueness statement concerning the isomorphism classes occurring in the disintegration Theorem~\ref{thm:disintegrating_space_actions}. Given the subtleties necessary in the study of Type I groups and $\Cstar$-algebras in the Hilbert space context, it does not seem to be realistic to strive for such a result at this moment.

The second main step consists of removing the hypothesis that the given representation of $G$ on $\Ell^p(\ts,\mu)$ originate from an action on the underlying probability space $(\ts,\mu)$. Under mild conditions, it can be shown that an action of $G$ on $\Ell^p(\ts,\mu)$ as isometric lattice automorphisms that leave the constants fixed, can be transferred to another model where there \emph{is} such an underlying action; see Theorem~\ref{thm:transfer}. We are then back in the ergodic theoretical context, and combination with the result from the first main step yields a disintegration result for these representations into order indecomposable representations as well. The pertinent Theorem~\ref{thm:disintegrating_Markov_actions} should be thought of as an ordered relative of the general unitary disintegration result \cite[Theorem~18.7.6]{Dixmier77}. The key transfer Theorem~\ref{thm:transfer} for this step is strongly inspired by the material in \cite{EiFaHaNa15}, and it is a pleasure to thank Markus Haase for drawing our attention to this.

It seems that, for practical purposes, our main results have a rather broad range of validity; we will now make a few technical remarks to support this statement. One of the re-occurring hypotheses in this paper is that a space be Polish (i.e.\ separable and metrizable in a complete metric). For a locally compact Hausdorff space, being Polish is equivalent to being second countable; see \cite[Theorem~5.3]{Kechris95}. Thus all Lie groups are Polish (for a more extensive list of Polish groups see \cite[Section~1.3]{Becker-Kechris96}), and, more generally, so are all differentiable manifolds. Therefore, the factorization Theorem~\ref{thm:factorization} and the disintegration Theorem~\ref{thm:disintegrating_space_actions}\textemdash for which the underlying Polish space $\ts$ need not even be locally compact\textemdash are applicable to all actions of Lie groups on differentiable manifolds. In a similar vein, we note that it follows from the combination of \cite[Vol. I, Exercise~1.12.102]{Bogachev07} and \cite[Vol. II, Example~6.5.2]{Bogachev07} that the measure space $(\ts,\mu)$ is always separable whenever $\ts$ is a separable metric space and $\mu$ is a Borel probability measure on $\ts$. Therefore, the disintegration Theorem~\ref{thm:disintegrating_Markov_actions}, where this separability is assumed, covers several commonly occurring situations as well.

This paper is organized as follows.

In Section~\ref{sec:notation_and_terminology}, we introduce some terminology and notation, and establish a few preliminary results on order indecomposability and strong continuity of canonical representations of groups on $\Ell^p$-spaces.

The first part of Section~\ref{sec:direct_integrals} is concerned with an extension of part of the theory of direct integrals of Banach spaces and Banach lattices in \cite{HaLeRa91}. The measurable families of norms figuring in \cite{HaLeRa91} are not sufficient for our context, where a measurable family of semi-norms occurs naturally. Moreover, our measures need not be complete. We generalize the theory accordingly. After that, $\Ell^p$-direct integrals of representations are introduced, and possible perspectives in representation theory are briefly discussed. The usual direct integrals of representations on separable Hilbert spaces are shown to be special cases of the general formalism.

Section~\ref{sec:disintegrating_space_actions} contains the results of the first main step, i.e.\ the factorization Theorem~\ref{thm:factorization} and the disintegration Theorem~\ref{thm:disintegrating_space_actions} in the case of an action on the underlying measure space. As a worked example, we give a concrete disintegration of the representations of the unit circle on the $\Ell^p$-spaces of the closed unit disk, as these are canonically associated with the action of the circle on this disk as rotations.

Section~\ref{sec:disintegrating_Markov_actions} is concerned with disintegrating representations when there is (initially) no action on an underlying measure space. Its main result, the disintegration Theorem~\ref{thm:disintegrating_Markov_actions}, is our ordered relative of the general unitary disintegration in \cite[Theorem~18.7.6]{Dixmier77}.

Section~\ref{sec:perspective} contains some remarks on the current status of the theory and on possible further developments.

\medskip
\emph{Reading guide.} Even though this paper was motivated by a representation theoretical question in an ordered context (as is reflected in the terminology of the present section), the interpretation of the main results as answers to this question is almost just an afterthought. The reader can find definitions and terminology concerning vector lattices in e.g.\ \cite{Zaanen}, but, if so desired, the limited number of occurrences of this terminology in the sequel that go beyond the notions of a vector lattice and a lattice homomorphism can also safely be ignored. The paper can then be read from a primarily ergodic theoretical, functional analytical, or general representation theoretical perspective.

\section{Preliminaries}\label{sec:notation_and_terminology}

In this section, we fix terminology and notation, and establish a few preliminary results on group representations.

\subsection{Terminology and notation}
All vector spaces, except the Hilbert spaces in Section~\ref{Hilbert space direct integrals}, are over the real numbers. This is no essential restriction, as the results in this paper extend to complex $\Ell^p$-spaces and (in Section~\ref{sec:direct_integrals}) to complex Banach spaces and Banach lattices in an obvious manner, but this convention reduces the necessary terminology and size of the proofs.

Topological spaces are not assumed to be Hausdorff. A topological space is called locally compact if every point has an open neighbourhood with compact closure.

If $\ts$ is a topological space, then $\Cc(\ts)$ and $\Cb(\ts)$ denote the continuous functions on $\ts$ that have compact support and that are bounded, respectively.

Topological groups are groups for which inversion is continuous and multiplication is continuous in two variables simultaneously. They are not assumed to be Hausdorff or locally compact.

A \emph{topological dynamical system} is a pair $(G,\ts)$, where the topological group $G$ acts as homeomorphisms on the topological space $\ts$ such that the map $(g,x)\mapsto gx$ is continuous from $G\times \ts$ to $\ts$. The system is called Polish if both $G$ and $\ts$ are Polish.

A measure on a $\sigma$-algebra is $\sigma$-additive and takes values in $[0,\infty]$. It is not assumed to be $\sigma$-finite. If $\ts$ is a topological space, then a Borel measure is a measure on the Borel $\sigma$-algebra of $\ts$, without any further assumptions.

For $(\ts,\mu)$ a measure space and $1\leq p\leq\infty$, $\La^{p}(\ts,\mu)$ denotes the semi-normed space of all $p$-integrable extended functions $f:\ts\to\R\cup\{-\infty,\infty\}$, and $\Ell^{p}(\ts,\mu)$ denotes the Banach lattice of all equivalence classes of extended functions $f\in\La^{p}(\ts,\mu)$, under $\mu$-almost everywhere equality. We will often work with an extended function $f$ that is an element of $\La^{p}(\ts,\mu)$ for different measures $\mu$ on $\ts$, and we will consider the equivalence classes of $f$ in $\Ell^{p}(\ts,\mu)$ for these $\mu$. It is essential to keep a clear distinction between these objects, so (with the exception of Section~\ref{sec:disintegrating_Markov_actions}) we do not identify functions that are equal almost everywhere, and, when $p$ is fixed, we denote the equivalence class in $\Ell^{p}(\ts,\mu)$ of an element $f\in\La^{p}(\ts,\mu)$ by $[f]_{\mu}$.

In the same vein, if $\vs$ is a vector space, $\w$ is an index, and ${\norm{\,\cdot\,}}_{\w}$ is a semi-norm on $\vs$, then we denote the equivalence class of $x\in\vs$ in $\vs/\ker({\norm{\,\cdot\,}}_{\w})$  by $[x]_{\w}$.

If $Y$ is a subset of $\ts$, then $\ind_Y$ is the characteristic function of $Y$ on $\ts$.

If $B$ is a normed space, then $\La(B)$ denotes the bounded linear operators on $B$.

\subsection{Preliminaries on group representations}\label{subsec:group_actions}

Suppose that the abstract group $G$ acts as measure preserving transformations on the measure space $(\ts,\mu)$. We then say that $\mu$ is a \emph{$G$-invariant measure}. In this case, for every $1\leq p<\infty$, $\rep_\mu(g)[f]_\mu:=[x\mapsto f(g^{-1}x)]_\mu$ is a well-defined representation of $G$ as isometric lattice isomorphisms of $\Ell^p(\ts,\mu)$. We will refer to this representation (and to similarly defined representations on other function spaces) as the \emph{canonical representation} on $\Ell^p(\ts,\mu)$; in the literature this is also called a Koopman representation.

A measurable subset $Y$\ of $\ts$ is \emph{$\mu$-essentially $G$-invariant} if $\mu(g Y\!\mathbin{\Delta} Y)=0$ for all $g\in G$, where $Y\mathbin{\Delta} g Y:=(Y\cup g Y)\setminus(Y\cap g Y)$ is the symmetric difference of $Y$ and $g Y$. An \emph{ergodic measure} on $\ts$ is a $G$-invariant measure $\mu$ such that $\mu(Y)=0$ or $\mu(Y)=1$ for each $\mu$-essentially $G$-invariant measurable subset $Y$ of $\ts$.

We will now investigate the relationship between the ergodicity of the measure $\mu$ and the order indecomposability of $\rep_{\mu}:G\to\La(\Ell^{p}(X,\mu))$. This is essential for the representation theoretical interpretation of our disintegration results, but not for these results as such, so that the reader with a primarily ergodic theoretical or functional analytic nterest can skip the next two results. We need the following lemma, which follows easily from~\cite[p.~44]{Zaanen}.

\begin{lemma}\label{lem:bands_in_Lp}
Let $(\ts,\mu)$ be a $\sigma$-finite measure space, and let $1\leq p\leq\infty$.
If $Y\subseteq\ts$ is measurable, let
\begin{align*}
B_Y=\left\{[f]_\mu\in\Ellp(\ts,\mu) :  f(y)=0 \textrm{ for $\mu$-almost all }y\in Y\right\}.
\end{align*}
Then $B_Y $ is a projection band in $\Ellp(\ts,\mu)$, and all projection bands in $\Ellp(\ts,\mu)$ are of this form.
If $Y_1$ and $Y_2$ are measurable subsets of $\ts$, then $B_{Y_1}=B_{Y_2}$ if and only if $\mu(Y_1\Delta Y_2)=0$.
\end{lemma}

Recall that the measure algebra $\measalg$ of $(X,\mu)$ consists of the equivalence classes $[Y]_\mu$ of measurable subsets $Y$ of $\ts$, where $Y_1$ and $Y_2$ are equivalent when $\mu(Y_1\Delta Y_2)=0$. Lemma~\ref{lem:bands_in_Lp} shows that there is a bijection between the elements of $\measalg$ and the projection bands in $\Ell^{p}(\ts,\mu)$, where an element of $[Y]_\mu$ of the measure algebra corresponds to the well-defined band $B_{[Y]_\mu}:=B_{Y}$.

If an abstract group $G^\prime$ acts as positive operators on $\Ell^{p}(\ts,\mu)$, then it permutes the projection bands in $\Ell^{p}(\ts,\mu)$. If, as is the case for our group $G$, this positive action originates canonically from an action as  measure preserving transformations on $(\ts,\mu)$, then $G$ also acts canonically on $\measalg$: for $g\in G$ and $[Y]\in\measalg$, the action $g[Y]_\mu:=[gY]_\mu$ is well-defined. These two actions are compatible with the map $[Y]_\mu\mapsto B_{[Y]_\mu}$. This is the content of part \eqref{prop:characterization_band_irreducibility_1} of the next result, and it is exploited in parts~\eqref{prop:characterization_band_irreducibility_2},~\eqref{prop:characterization_band_irreducibility_3}, and~\eqref{prop:characterization_band_irreducibility_4}.

\begin{proposition}\label{prop:characterization_band_irreducibility}
Let $G$ be an abstract group, acting as measure preserving transformations on a $\sigma$-finite measure space $(\ts,\mu)$, and let $1\leq p\leq \infty$.
\begin{enumerate}
\item\label{prop:characterization_band_irreducibility_1} If  $[Y]_\mu\in\measalg$, and $B_{[Y]_\mu}$ is the corresponding projection band in $\Ell^p(\ts,\mu)$, then $\rep_{\mu}(g)B_{[Y]_\mu}=B_{g[Y]_\mu}$ \textup{(}$g\in G$\textup{)}.
\item\label{prop:characterization_band_irreducibility_2} For $g\in G$, the projection bands in $\Ell^p(\ts,\mu)$ that are fixed by $g$ correspond to the fixed points of $g$ in $\measalg$.
\item\label{prop:characterization_band_irreducibility_3} The $G$-invariant projection bands in $\Ell^p(\ts,\mu)$ correspond to the fixed points of $G$ in $\measalg$.
\item\label{prop:characterization_band_irreducibility_4} The canonical representation $\rep_{\mu}:G\to\La(\Ell^{p}(\ts,\mu))$ of $G$ as isometric lattice automorphisms on $\Ell^p(\ts,\mu)$ is order indecomposable if and only if $\mu$ is ergodic.
\end{enumerate}
\end{proposition}

\begin{proof}
As for \eqref{prop:characterization_band_irreducibility_1}, let $[f]_\mu\in B_{[Y]_\mu}$, so that $\mu(\{Y\cap\,\supp f\})=0$. By the invariance of $\mu$, we have $\mu(\{gY\cap\, g\,\supp f\})=0$. Since $g\,\supp f=\supp\,gf$, we see that $\mu(\{gY\cap\, \supp \,gf\})=0$, i.e.\ $\rep_{\mu}(g)[f]_\mu\in B_{gY}=B_{[gY]_\mu}=B_{g[Y]_\mu}$. Hence $\rep_{\mu}(g)B_{[Y]_\mu}\subseteq B_{g[Y]_\mu}$. Then also $\rep_{\mu}(g)^{-1} B_{g[Y]_\mu}=\rep_{\mu}(g^{-1})B_{[gY]_\mu}\subseteq B_{g^{-1}[gY]_\mu)}=B_{[Y]_\mu}$, so that $B_{g[Y]_\mu}\subseteq \rep_{\mu}(g)B_{[Y]_\mu}$.

The parts~\eqref{prop:characterization_band_irreducibility_2} and~\eqref{prop:characterization_band_irreducibility_3} are immediate from \eqref{prop:characterization_band_irreducibility_1}.

As for \eqref{prop:characterization_band_irreducibility_4}, we know from \eqref{prop:characterization_band_irreducibility_2} that $\rep_{\mu}$ is order indecomposable if and only if $[\emptyset]_\mu$ and $[X]_\mu$ are the only points of $\measalg$ that are fixed by the $G$-action.  The latter condition is equivalent to the ergodicity of $\mu$.
\end{proof}

As a further preliminary, we will now investigate the strong continuity of canonical representations of topological groups on spaces of continuous functions with compact support and on $\Ell^p$-spaces, the latter being our principal point of interest. The matter is usually considered in the context of a locally compact Hausdorff group and a locally compact Hausdorff space (see e.g.~\cite[p.~68]{Folland95}), but more can be said. The results clarify natural questions concerning our context (see e.g.\ Corollary~\ref{cor:automatic_strong_continuity_Polish_pair}), and, in view of possible future study of canonical group actions on $\Ell^p$-spaces, this seems a natural moment to collect a few basic facts in a sharp formulation.

A reference for the following result would be desirable, but we are not aware of one for the statement in this generality. The left and right uniform continuity of compactly supported continuous functions on a locally compact Hausdorff group (see\ \cite[Proposition~2.6]{Folland95}) are special cases.

\begin{lemma}\label{lem:uniform_continuity_sup_norm}
Let $(G,\ts)$ be a topological dynamical system. Then the canonical representation $\rep$ of $G$ as isometric lattice automorphisms of $(\Cc(\ts),{\Vert\,\cdot\,\Vert}_\infty)$ is strongly continuous.
\end{lemma}

\begin{proof}
It is sufficient to prove that $g\mapsto\rep(g)f$ is continuous at~$e$ for each $f\in \Cc(\ts)$. Let $\epsilon>0$. For each $x\in \ts$, there exist a symmetric open neighbourhood~$U_x$ of~$e$ in~$G$ and an open neighbourhood~$V_x$ of~$x$ in~$\ts$ such that $|f(g^{-1}y)-f(x)|<\epsilon/2$ for all $g\in U_x$ and $y\in V_x$. Let $\bigcup_{i=1}^n V_{x_i}$ be a finite cover of $\supp f$, and put $U=\bigcap_{i=1}^n U_{x_{i}}$. If $x\in\supp f$, say $x\in V_{x_{i_0}}$, and $g\in U\subseteq U_{x_{i_0}}$, then $|f(g^{-1}x)-f(x)|\leq |f(g^{-1}x)-f(x_{i_0})| + |f(x)-f(x_{i_0})|<\epsilon/2 + \epsilon /2 =\epsilon$. Since~$U$ is symmetric, we also have $|f(gx)-f(x)|<\epsilon$ for all $x\in\supp f$ and $g\in U$. Therefore, if $g\in U$ and $x\in\ts$ are such that $g^{-1}x\in\supp f$, we have $|f(g^{-1}x)-f(x)|=|f(g(g^{-1}x))-f(g^{-1}x)|<\epsilon$. We have now shown that, for $g\in U$, $|f(g^{-1}x)-f(x)|<\epsilon$ whenever $x\in\supp f$ or $g^{-1}x\in\supp f$. Since $|f(g^{-1}x)-f(x)|=0$ for all remaining~$x$, we are done.
\end{proof}

\begin{proposition}\label{prop:strong_continuity_basic}
Let $(G,\ts)$ be a topological dynamical system, and assume that $G$ is locally compact. Let~$\mu$ be a Borel measure on~$\ts$ that is finite on compact subsets of~$\ts$. Then, for $1\leq p<\infty$, the canonical representation $\rep_{\mu}$ of~$G$ as possibly unbounded lattice automorphisms of $(\Cc(X),{\Vert\,\cdot\,\Vert}_p)$ is strongly continuous. If~$\mu$ is $G$-invariant, then the canonical representation~$\rep_{\mu}$ of~$G$ as isometric lattice automorphisms of the closure of $\Cc(X)$ in $\Ell^p(X,\mu)$ is strongly continuous.
\end{proposition}

\begin{proof}
Let $f\in \Cc(\ts)$, $g_0\in G$, and $\epsilon>0$ be given. Choose an open neighbourhood $V$ of $e$ in $G$ with compact closure. Then $g_0\overline V\cdot\supp f$ is compact, hence has finite measure. By Lemma~\ref{lem:uniform_continuity_sup_norm}, there exist an open neighbourhood $U$ of $e$ in $G$ such that $\mu(g_0\overline V\cdot\supp f){\Vert\rep_\mu(g)f-\rep_\mu(g_0)f\Vert}_\infty^p<\epsilon^p$ for all $g\in g_0U$. We may assume that $U\subseteq V$. Then, for $g\in g_0U$,
\begin{align*}
{\Vert \rep_\mu(g)f-\rep_\mu(g_0)f\Vert}_p^p&=\int_\ts {|(\rep_\mu(g)f)(x)-(\rep_\mu(g_0)f)(x)|}^p\,\ud\mu(x)\\
&=\int_{g\,\supp f\cup g_0\,\supp f}{|(\rep_\mu(g)f)(x)-(\rep_\mu(g_0)f)(x)|}^p\,\textup{d}\mu(x)\\
&\leq \int_{g_0\overline V\cdot\supp f}{|(\rep_\mu(g)f)(x)-(\rep_\mu(g_0)f)(x)|}^p\,\ud\mu(x)\\
&\leq \int_{g_0\overline V\cdot\supp f} {\Vert\rep_\mu(g)f-\rep_\mu(g_0)f\Vert}_\infty^p\,\ud\mu(x)\\
&<\epsilon^p.
\end{align*}
The final statement follows from a $3\epsilon$-argument.
\end{proof}

Proposition~\ref{prop:strong_continuity_basic} points at the heart of the matter: under a mild condition on the $G$-invariant Borel measure $\mu$, the natural representation subspace for $G$ in $\Ell^p(\ts,\mu)$ is the closure of $\Cc(X)$. In some cases, this closure equals  $\Ell^p(X,\mu)$, and we include this well-known result for the sake of completeness. For this, recall (see~\cite[Definition~18.4]{Aliprantis-Burkinshaw98}) that a Borel measure $\mu$ on a locally compact Hausdorff space is said to be \emph{regular} if
\begin{equation*}
\mu(K)<\infty\textup{ for all compact subsets }K\textup{ of }\ts,
\end{equation*}
\begin{equation*}
\mu(Y)=\inf\{\mu(V): Y\subseteq V,\,V \textup{ open}\}\textup{ for all Borel subsets }Y\textup{ of }\ts,
\end{equation*}
and
\begin{equation*}
\mu(V)=\sup\{\mu(K) : K\subseteq V,\,K\textup{ compact}\}\textup{ for all open subsets }V\textup{ of }\ts.
\end{equation*}
For such a measure, $\Cc(\ts)$ is dense in  $\Ell^p(\ts,\mu)$; see~ \cite[Theorem~31.11]{Aliprantis-Burkinshaw98}. Combination with Proposition~\ref{prop:strong_continuity_basic} gives the following, generalizing the well-known fact that the left and right regular representations of a locally compact Hausdorff group $G$ on $\Ell^p(G)$ are strongly continuous for $1\leq p<\infty$; see \cite[Proposition~2.41]{Folland95}.

\begin{corollary}\label{cor:strong_continuity_Lp_regular_measures}
Let $(G,\ts)$ be a topological dynamical system, and assume that $G$ is locally compact and that $\ts$ is a locally compact Hausdorff space. Let $\mu$ be a $G$-invariant regular Borel measure on $\ts$. Then, for $1\leq p<\infty$, the canonical representation $\rep_{\mu}$ of $G$ as isometric lattice automorphisms of $\Ell^p(X,\mu)$ is strongly continuous.
\end{corollary}

Although \cite[p.~68]{Folland95}\textemdash where it is assumed that $G$ is Hausdorff\textemdash mentions that the above result holds, and it is likewise stated\textemdash for locally compact second countable Hausdorff $G$ and $X$\textemdash without proof on \cite[p.~875]{Nevo06}, we are not aware of a reference for an actual proof of this folklore result, nor for one of the more basic Proposition~\ref{prop:strong_continuity_basic}.

If $G$ and $X$ are not both locally compact, the proof of the strong continuity in Corollary~\ref{cor:strong_continuity_Lp_regular_measures} breaks down. However, there is an alternative context where a similar result can still be established along similar lines.

\begin{lemma}\label{lem:Cb_density}
Let $\ts$ be a metric space, and let $\mu$ be a Borel probability measure on $\ts$. Then $\Cb(\ts)$ is dense in $\Ell^p(X,\mu)$ for $1\leq p<\infty$.
\end{lemma}

\begin{proof}
It is sufficient to approximate the characteristic function $\ind_Y$ of an arbitrary Borel subset $Y$ of $X$ by elements of $\Cb(X)$. Since we know from \cite[Theorem~17.10]{Kechris95} that, for every Borel subset $Y$ of $X$, $\mu(Y)=\inf\{\mu(U) : Y\subseteq U,\,U \textup{ open} \}$, it is already sufficient to approximate $\ind_U$ for an arbitrary open subset $U$ of $\ts$. We may assume that $U\neq X$. In that case, let $f_n(x)=\min(1,nd(x,U^c))$ ($n=1,2,\ldots$). Then $f_n\in \Cb(X)$ and $0\leq f_n\uparrow\ind_U$, so that ${\Vert f_n-\ind_U\Vert}_p\to 0$ as $n\to\infty$ by the dominated convergence theorem.
\end{proof}

\begin{proposition}\label{prop:strong_continuity_action_of_first_countable_group}
Let~$G$ be a first countable group, acting as Borel measurable transformations on a metric space $\ts$ with a $G$-invariant Borel probability measure~$\mu$. Suppose that, for all $x\in \ts$, the map $g\mapsto gx$ is continuous from~$G$ to~$\ts$. Then, for $1\leq p<\infty$, the canonical representation $\rep_{\mu}$ of~$G$ as isometric lattice automorphisms of $\Ell^p(X,\mu)$ is strongly continuous.
\end{proposition}

\begin{proof}
In view of Lemma~\ref{lem:Cb_density} and  a $3\epsilon$-argument, it is sufficient to prove that the map $g\mapsto\rep_{\mu}(g)f$ is continuous for each $f\in \Cb(X)$. Since $G$ is first countable, continuity at a point $g\in G$ is the same as sequential continuity at $g$. If $g_n\to g$ as $n\to\infty$, then $g_n f\to gf$ pointwise as $n\to\infty$, by the continuity assumption on the $G$-action and the continuity of $f$. The dominated convergence theorem then yields that ${\Vert\rep_{\mu}(g_n)f-\rep_{\mu}(g)f \Vert}_p\to 0$ as $n\to\infty$.
\end{proof}

As a very special case, let us explicitly mention the strong continuity of the representation in Section~\ref{sec:disintegrating_space_actions}.

\begin{corollary}\label{cor:automatic_strong_continuity_Polish_pair}
Let $(G,\ts)$ be a Polish topological dynamical system, and suppose that $\mu$ is a $G$-invariant Borel probability measure on $\ts$. Then, for each $1\leq p<\infty$, the canonical representation $\rep_{\mu}$ of $G$  as isometric lattice automorphisms of $\Ell^p(X,\mu)$ is strongly continuous.
\end{corollary}

\begin{remark} Every Borel probability measure on a Polish space is regular; this follows from \cite[Theorem~17.10]{Kechris95} and \cite[Vol. II, Theorem~7.1.7]{Bogachev07}. However, since local compactness of $G$ and $\ts$ are not assumed in Corollary~\ref{cor:automatic_strong_continuity_Polish_pair}, Corollary~\ref{cor:strong_continuity_Lp_regular_measures} is still not applicable here.
\end{remark}

\section{$\kstext$-direct integrals of Banach spaces and representations}\label{sec:direct_integrals}

This section provides the framework for the disintegration Theorems~ \ref{thm:disintegrating_space_actions} and~ \ref{thm:disintegrating_Markov_actions}.

We start by defining $\kstext$-direct integrals of Banach spaces and Banach lattices in the spirit of \cite[Sections 6.1 and 7.2]{HaLeRa91}. The idea in \cite{HaLeRa91} is, roughly, to begin with a `core' vector space $V$ that is supplied with a (suitable) family of norms ${\Vert\,\cdot\,\Vert}_\w$, depending on the points $\w$ of a measure space $(\Omega,\nu)$. If $\{B_\w\}_{\w\in\Omega}$ is the corresponding family of Banach space completions of $V$, then one can consider sections from $\Omega$ to $\bigsqcup_{\w\in\Omega} B_{\w}$. There is a natural notion of measurable sections, and the $B_\w$ are `glued together' by restricting attention to measurable sections and identifying measurable sections that are $\nu$-almost everywhere equal. For any K\"othe space $E$ associated with $(\Omega,\nu)$, one can then require, for a measurable section $s$, that the function $\w\mapsto\norm{s(\omega)}$ be in $E$. If $E$ satisfies appropriate additional properties, then the equivalence classes of such sections form a Banach space, which is called the $E$-direct integral of the family $\{B_\omega\}_{\w\in\Omega}$.

In Section \ref{subsec:direct_integrals_of_Banach_spaces}, this program is carried out for $E=\Ell^{p}(\Omega,\nu)$ ($1\leq p<\infty$), but with two noticeable modifications as compared to \cite{HaLeRa91}. The first is that the family of norms figuring in \cite[p.~61]{HaLeRa91} is replaced with a family of semi-norms. The need for this comes up quite naturally in our context, and it seems to the authors that this may also be the case elsewhere. The second difference is that the measure $\nu$ need not be complete. Completeness of measures is the standing assumption in \cite[p.~5]{HaLeRa91}, but the measure we will apply the formalism to in Section~\ref{sec:disintegrating_space_actions} need not be complete. One has to be extra cautious then, and particularly in a vector-valued context; the proof of Proposition~\ref{prop:direct_integral_is_complete} may serve as evidence for this. Since, in addition, the proofs\textemdash in a technically more convenient context\textemdash in \cite{HaLeRa91} of the type of results that we need are sometimes more indicated than given in full, we found it necessary to give the applicable details for our case, rather than refer the reader to \cite{HaLeRa91} with the task to make the pertinent modifications. As a consequence of this choice of presentation,  we are also able to give a precise discussion of the relation with the Bochner integral (see~Remark~\ref{rem:Bochner}) and with the usual theory of direct integrals of separable Hilbert spaces  and of decomposable operators (see~Section~\ref{Hilbert space direct integrals}), proving that these are particular cases of the general theory in this section.

In Section~\ref{subsec:direct_integrals_of_representations}, we define decomposable operators and the $\kstext$-direct integral of a decomposable family of representations of a group $G$, which is a representation of $G$ on the $\kstext$-direct integral of Banach spaces from Section~\ref{subsec:direct_integrals_of_Banach_spaces}. One way to obtain such a decomposable family of representations is when it originates from one common `core' representation $\widetilde\repint$ of $G$ on the `core' vector space $\vs$. Even though it is all fairly natural, we are not aware of previous similar work in the context of (dis)integrating representations.

As shown in Section~\ref{Hilbert space direct integrals}, the framework in Section~\ref{subsec:direct_integrals_of_representations} includes the usual theory of direct integrals of (representations on) separable Hilbert spaces.

Finally, in Section~\ref{subsec:perspective}, we sketch a perspective that a more or less obvious extension of the formalism could have in representation theory.


\subsection{$\boldsymbol{\Ell^p}$-direct integrals of Banach spaces}\label{subsec:direct_integrals_of_Banach_spaces}
We will now define $\kstext$-direct integrals of a suitable family of Banach spaces. These are Banach spaces that generalize the Bochner $\Ell^p$-spaces (see Remark~\ref{rem:Bochner}) and the direct integrals of separable Hilbert spaces (see Section~\ref{Hilbert space direct integrals}).

Let $(\Omega,\nu)$ be a measure space, and let $\vs$ be a vector space. For clarity, let us recall that our measures need not be finite (or even $\sigma$-finite) or complete. We say that a collection $\left\{{\norm{\,\cdot\,}}_{\w}\right\}_{\w\in\Omega}$ is a \emph{measurable family of semi-norms on $\vs$} if ${\norm{\,\cdot\,}}_{\w}$ is a semi-norm on $\vs$ for each $\w\in\Omega$, and $\omega\mapsto {\norm{x}}_{\w}$ is a measurable function on $\Omega$ for each $x\in \vs$. For later use, let us record that this is the same as requiring that the (identical) function $\omega\mapsto {\norm{[x]_{\w}}}_{\w}$ is a measurable function on $\Omega$ for all $x\in \vs$, where ${\norm{[x]_{\w}}}_{\w}$ is the value of the induced norm ${\norm{\,\cdot\,}}_{\w}$ on $\vs/\ker({\norm{\,\cdot\,}}_{\w})$ at the equivalence class $[x]_\w$ of $x$ in $\vs/\ker({\norm{\,\cdot\,}}_{\w})$.

Let $\left\{\bs_{\w}\right\}_{\w\in\Omega}$ be a collection of Banach spaces and suppose that $\left\{{\norm{\,\cdot\,}}_{\w}\right\}_{\w\in\Omega}$ is a measurable family of semi-norms on $V$ such that, for each $\w\in\Omega$, $\bs_{\w}$ is the Banach space completion of $\vs/\ker({\norm{\,\cdot\,}}_{\w})$ with respect to the induced norm ${\norm{\,\cdot\,}}_{\w}$ on $\vs/\ker({\norm{\,\cdot\,}}_{\w})$. Then we say that $\left\{\bs_{\w}\right\}_{\w\in\Omega}$ is a \emph{measurable family of Banach spaces over $(\Omega,\nu,V)$}. For conciseness, we usually do not explicitly mention the specific measurable family of semi-norms $\left\{{\norm{\,\cdot\,}}_{\w}\right\}_{\w\in\Omega}$ on $V$ that gives rise to $\left\{\bs_{\w}\right\}_{\w\in\Omega}$, as this family will generally be clear from the context.

Analogously, suppose that $V$ is a vector lattice and that $\left\{\norm{\,\cdot\,}_{\w}\right\}_{\w\in\Omega}$ is a measurable family of lattice semi-norms on $\vs$ such that, for each $\w\in\Omega$, $\bs_{\w}$ is the Banach lattice completion of $\vs/\ker({\norm{\,\cdot\,}}_{\w})$ with respect to the induced lattice norm ${\norm{\,\cdot\,}}_{\w}$ and the induced ordering on $\vs/\ker({\norm{\,\cdot\,}}_{\w})$. Then we say that a family $\left\{\bs_{\w}\right\}_{\w\in\Omega}$ of Banach lattices is a \emph{measurable family of Banach lattices over $(\Omega,\nu,V)$}. When using this terminology, we will tacitly assume that $V$ is a vector lattice, and that the ${\norm{\,\cdot\,}}_\w$ are lattice semi-norms.

Let $\left\{\bs_{\w}\right\}_{\w\in\Omega}$ be a measurable family of Banach spaces over $(\Omega,\nu,V)$. We say that a map $\sect:\Omega\to\bigsqcup_{\w\in\Omega}\bs_{\w}$ is a \emph{section of $\{\bs_{\w}\}_{\w\in\Omega}$} if $\sect(\w)\in\bs_{\w}$ for each $\w\in\Omega$. A \emph{simple section} is a section $\sect$ for which there exist $n\in\N$, $x_{1},\ldots, x_{n}\in \vs$, and measurable subsets $A_{1},\ldots, A_{n}$ of $\Omega$ such that $\sect(\w)=\left[\sum_{k=1}^{n}\ind_{A_{k}}(\w) x_{k}\right]_{\w}$ for all $\w\in\Omega$. Choosing the $A_k$ disjoint, we have ${\norm{\sect(\w)}}_{\w}=\sum_{k=1}^{n}\ind_{A_{k}}(\w){\norm {[x_{k}]_{\w}}}_{\w}$, so that the function $\w\mapsto{\norm{\sect(\w)}}_{\w}$ on $\Omega$ is measurable for each simple section $\sect$.

 A section $\sect$ of $\left\{\bs_{\w}\right\}_{\w\in\Omega}$ is said to be \emph{measurable} if there exists a sequence $(\sect_{k})_{k=1}^{\infty}$ of simple sections such that, for all $\w\in\Omega$,
 $\sect_k(\w)\to\sect(\w)$ in $\bs_{\w}$ as $k\to\infty$. Then clearly ${\norm{\sect_k(\w)}}_\omega\to{\norm{\sect(\w)}}_{\w}$ for all $\w\in\Omega$ as $k\to\infty$, and hence,  as a consequence of the measurability of the functions $\w\mapsto{\norm{\sect_k(\w)}}_{\w}$ on $\Omega$, the function $\w\mapsto{\norm{\sect(\w)}}_{\w}$ is a measurable function on $\Omega$ for each measurable section $\sect$. The measurable sections form a vector space, and we will denote the section that maps every $\w\in\Omega$ to the zero element of $B_{\w}$ simply by $0$. We also note that, if $A$ is a measurable subset of $ \Omega$ and $\sect$ is a simple section, then $\ind_A s$ (defined in the obvious pointwise way) is again a simple section. It follows easily from this that the measurable sections are a module over the measurable functions on $\Omega$ under pointwise operations.

We define the \emph{direct integral $\int_{\Omega}^{\oplus}\bs_{\w}\,\ud\nu(\w)$ of $\{\bs_{\w}\}_{\w\in\Omega}$ with respect to $\nu$} to be the space of all equivalence classes $[\sect]_{\nu}$ of measurable sections $\sect$ of $\{\bs_{\w}\}_{\w\in\Omega}$, where two measurable sections are equivalent if they agree $\nu$-almost everywhere on $\Omega$. We say that the $\bs_{\w}$ are the \emph{fibers} of $\int_{\Omega}^{\oplus}\bs_{\w}\,\ud\nu(\omega)$, and we introduce a vector space structure on $\int_{\Omega}^{\oplus}\bs_{\w}\,\ud\nu(\omega)$  in the usual representative-independent way.

If $\left\{\bs_{\w}\right\}_{\w\in\Omega}$ is a measurable family of Banach lattices over $(\Omega,\nu,V)$, then, in addition, we can meaningfully define a natural partial ordering on $\int_{\Omega}^{\oplus}\bs_{\w}\,\ud\nu(\omega)$ by
\begin{align*}
[\sect]_{\nu}\geq [t]_{\nu}\quad\Leftrightarrow \quad \sect(\w)\geq t(\w) \textrm{ for $\nu$-almost all $\w\in\Omega$}
\end{align*}
for $[\sect]_{\nu},[t]_{\nu}\in\int_{\Omega}^{\oplus}\bs_{\w}\,\ud\nu(\omega)$. Then $\int_{\Omega}^{\oplus}\bs_{\w}\,\ud\nu(\omega)$ is an ordered vector space. In fact, it is a vector lattice. For the latter statement, note that the pointwise supremum and infimum of two measurable sections are measurable again, as a consequence of the continuity of the lattice operations in each $\bs_{\w}$ and the fact that the pointwise supremum and infimum of two simple sections are simple sections again. It is then easily verified that, for $[\sect]_{\nu},\,[t]_{\nu}\in \int_{\Omega}^{\oplus}\bs_{\w}\,\ud\nu(\omega)$, $[\sect]_{\nu}\vee[t]_{\nu}$ exists in $\int_{\Omega}^{\oplus}\bs_{\w}\,\ud\nu(\omega)$, and that, in fact, $[\sect]_{\nu}\vee[t]_{\nu}=[\sect\vee t]_{\nu}$, where $[\sect\vee t]_{\nu}\in\int_{\Omega}^{\oplus}\bs_{\w}\,\ud\nu(\omega)$ is defined by $(\sect\vee t)(\omega):=\sect(\w)\vee t(\w)$ ($\w\in\Omega$). The expression for the infimum is similar.

For $p\in[1,\infty)$, we let the \emph{$\kstext$-direct integral $\left(\int_{\Omega}^{\oplus}\bs_{\w}\,\ud\nu(\w)\right)_{\ks}$ of $\{\bs_{\w}\}_{\w\in\Omega}$ with respect to $\nu$} be the subset of $\int_{\Omega}^{\oplus}\bs_{\w}\,\ud\nu(\w)$ consisting of those $[\sect]_{\nu}\in\int_{\Omega}^{\oplus}\bs_{\w}\,\ud\nu(\w)$ such that the function $\w\mapsto {\norm{\sect(\w)}}_{\w}$, which we know to be measurable, is in $\mathcal L^{p}(\Omega,\nu)$. This criterion is evidently independent of the particular representative $s$ of $[s]_{\nu}$, and we call $[\sect]_\nu$ and its representatives \emph{$p$-integrable \textup{(}with respect to $\nu$\textup{)}}. It follows from the triangle inequality for each ${\norm{\,\cdot\,}}_{\w}$ that $\left(\int_{\Omega}^{\oplus}\bs_{\w}\,\ud\nu(\w)\right)_{\ks}$ is a subspace of $\int_{\Omega}^{\oplus}\bs_{\w}\,\ud\nu(\w)$ and that
\begin{align}\label{eq:norm_definition}
{\norm{[\sect]_{\nu}}}_p:=\left(\int_\Omega {\left\Vert \sect(\w)\right\Vert}_\w^p\ud\nu\right)^{1/p}\quad([s]_\nu\in\left(\int_{\Omega}^{\oplus}\bs_{\w}\,\ud\nu(\w)\right)_{\ks})
\end{align}
defines a norm $[\sect]_{\nu}\mapsto{\norm{[\sect]_{\nu}}}_p$ on $\left(\int_{\Omega}^{\oplus}\bs_{\w}\,\ud\nu(\w)\right)_{\ks}$.

If $V$ is a vector lattice and $\left\{\bs_{\w}\right\}_{\w\in\Omega}$ is a measurable family of Banach lattices over $(\Omega,\nu,V)$, then it is easily verified that $\left(\int_{\Omega}^{\oplus}\bs_{\w}\,\ud\nu(\w)\right)_{\ks}$ is a vector sublattice of $\int_{\Omega}^{\oplus}\bs_{\w}(\w)\,\ud\nu$, and that \eqref{eq:norm_definition} supplies it with a lattice norm.

The $\kstext$-direct integrals of Banach spaces, as defined above, are, in fact, Banach spaces. To show this, we will (have to) use that the equivalence classes of the $p$-integrable simple sections are dense. This density, which is also a key ingredient of the proof of the disintegration Theorem~\ref{thm:disintegrating_space_actions}, is established in the following stronger result, based on a familiar truncation argument as in e.g.~\cite[proof of Proposition~2.16]{Ryan02}.

\begin{lemma}\label{lem:simple_sections_lie_dense}
Let $(\Omega,\nu)$ be a measure space, let $\vs$ be a vector space, and let $p\in[1,\infty)$. Let $\{\bs_{\w}\}_{\w\in\Omega}$ be a measurable family of Banach spaces over $(\Omega,\nu,\vs)$, and let $[\sect]_\nu\in \left(\int_{\Omega}^{\oplus}\bs_{\w}\,\ud\nu(\w)\right)_{\ks}$. Then, for each $\epsilon>0$, there exists a sequence $(\sect_k)_{k=1}^\infty$ of $p$-integrable simple sections such that ${\norm{\sect_k(\w)}}_{\w}\leq(1+\epsilon){\norm{\sect(\w)}}_\w$ for all $k\in\N$ and $\w\in\Omega$,  $\sect_k(\w)\to\sect(\w)$ in $\bs_{\w}$ as $k\to\infty$ for all $\w\in\Omega$, and ${\norm{[s]_{\nu}-[\sect_k]_{\nu}}}_{p}\to 0$ as $k\to\infty$.

If $\{\bs_{\w}\}_{\w\in\Omega}$ is a measurable family of Banach lattices over $(\Omega,\nu,\vs)$ and $[\sect]_\nu\geq 0$, then the sequence $(\sect_k)_{k=1}^\infty$ can be chosen such that, in addition, $[\sect_k]_{\nu}\geq 0$ for all $k\in\N$.
\end{lemma}

\begin{proof}
Let $[\sect]_{\nu}\in\left(\int_{\Omega}^{\oplus}\bs_{\w}\,\ud\nu(\w)\right)_\ks$. Then there exists a sequence $(\sect_{k}^\prime)_{k=1}^{\infty}$ of simple sections such that, for all $\w\in\Omega$, $\sect_k^\prime(\w)\to\sect(\w)$ in $\bs_{\w}$ as $k\to\infty$. For $k\in\N$, let $A_{k}:=\left\{\w\in\Omega : {\norm{\sect_{k}^\prime(\w)}}_{\w}\leq (1+\epsilon){\norm{\sect(\w)}}_{\w}\right\}$. Then $A_k$ is a measurable subset of $\Omega$, hence the section $\sect_{k}$, defined by $\sect_{k}:=\ind_{A_k}\sect_{k}^\prime$, is simple again. Furthermore, ${\norm{\sect_{k}(\w)}}_{\w}\leq (1+\epsilon){\norm{\sect(\w)}}_{\w}$ for all $k$ and all $\w\in\Omega$, so that each $s_k$ is $p$-integrable with respect to $\nu$. For all $\w\in\Omega$, $\sect_{k}(\w)\to\sect(\w)$ in $\bs_{\w}$ as $k\to\infty$. It then follows from the dominated convergence theorem that  $[\sect_{k}]_{\nu}\to[\sect]_{\nu}$ in $\left(\int_{\Omega}^{\oplus}\bs_{\w}\,\ud\nu(\w)\right)_\ks$.

For the second statement, suppose that $[\sect]_\nu\geq 0$, and let $\epsilon>0$. Choose a sequence $(s_{k}')_{k=1}^{\infty}$ of simple sections with the three properties in the first part of the lemma. There exists a measurable subset $A$ of $\Omega$ such that $\nu(A)=0$ and $s(\w)\geq 0$ for all $\w\in A^c$. Then the sequence $(s_{k})_{k=1}^{\infty}$, given by $\sect_{k}(\w):=\ind_{A^c}(\w)\sect^\prime_k(\w)^+ + \ind_A(\w)\sect^\prime_k(\w)$ for $k\in\N$ and $\w\in\Omega$, is as desired.
\end{proof}

We can now establish the completeness of $\kstext$-direct integrals of Banach spaces.

\begin{proposition}\label{prop:direct_integral_is_complete}
Let $(\Omega,\nu)$ be a measure space, let $\vs$ be a vector space, and let $1\leq p<\infty$. If $\{\bs_{\w}\}_{\w\in\Omega}$ is a measurable family of Banach spaces over $(\Omega,\nu,\vs)$, then $\left(\int_{\Omega}^{\oplus}\bs_{\w}\,\ud\nu(\w)\right)_{\ks}$ is a Banach space. If $\{\bs_{\w}\}_{\w\in\Omega}$ is a measurable family of Banach lattices over $(\Omega,\nu,\vs)$, then $\left(\int_{\Omega}^{\oplus}\bs_{\w}\,\ud\nu(\w)\right)_{\ks}$ is a Banach lattice.
\end{proposition}

\begin{proof}
Let $([s_k]_\nu)_{k=1}^\infty$ be a sequence in $\left(\int_{\Omega}^{\oplus}\bs_{\w}\,\ud\nu(\w)\right)_{\ks}$ such that $\sum_{k=1}^\infty{\norm{[s_k]_\nu}}_{p}<\infty$. Following the standard proof (see e.g.\ \cite[Theorem~6.6]{Folland99}), one shows that there exists a measurable subset $A$ of $\Omega$ such that $\nu(A^c)=0$ and $s(\w):=\lim_{n\to\infty}\sum_{k=1}^n \ind_A(\w)s_k(\w)$ exists for all $\w\in\Omega$. If one knew $s$ to be a measurable section, then the conclusion of the standard proof would show that the series  $\sum_{k=1}^\infty[s_k]_\nu$ converges to $s$. Now the pointwise limit of a sequence of scalar-valued measurable functions is measurable, and, more generally in the context of the Bochner integral, the limit of a sequence of strongly measurable functions is strongly measurable (a consequence of a version of the non-trivial Pettis measurability theorem; see e.g.\ \cite[Theorem~E.9]{Cohn93} for the latter). In our context, however, we have no such result. Fortunately, the following easily verified fact saves the day: If $X$ is a normed space and $Y$ is a dense subspace with the property that every absolutely convergent series with terms from $Y$ converges in $X$, then $X$ is a Banach space. With this and Lemma~\ref{lem:simple_sections_lie_dense} in mind, we see that it is sufficient to prove convergence of the series when the $s_k$ are \emph{simple} sections. In that case, $s$ is the pointwise limit of simple sections, hence is measurable by definition.
\end{proof}

\begin{remark}\label{rem:Bochner}\quad
\begin{enumerate}
\item
If $\vs$ is a Banach space with norm ${\norm{\,\cdot\,}}$, and if we take ${\norm{\,\cdot\,}}_\w=\norm{\,\cdot\,}$ for all $\w\in\Omega$, then all $B_\w$ equal $\vs$.  We claim that, in this case, $\left(\int_{\Omega}^{\oplus}\bs_{\w}\,\ud\nu(\w)\right)_\ks$  is the Bochner space $\Ell^p(\Omega,V,\nu)$ as it is defined for an arbitrary measure in \cite[Appendix~E]{Cohn93}. To see this non-trivial fact, note that a section $s$ is now a function $s:\Omega\to \vs$. It follows from \cite[Theorem~E.9]{Cohn93} (this is a version of the Pettis measurability theorem) and \cite[Proposition~E.2]{Cohn93}  that such a section is measurable in our terminology precisely if it is a strongly measurable function in the terminology of \cite[Appendix~E]{Cohn93}. Since the Bochner spaces in \cite[Appendix~E]{Cohn93} are defined (actually only for $p=1$, but this is immaterial), starting from the strongly measurable functions, in the same canonical fashion as $\left(\int_{\Omega}^{\oplus}\bs_{\w}\,\ud\nu\right)_\ks$ is defined, starting from the measurable sections, both spaces coincide.
\item Although it is usually not observed as such, the direct integrals of separable Hilbert spaces as they are defined in the literature are Bochner $\Ell^2$-spaces. This follows from part 1 of the current remark and Section~\ref{Hilbert space direct integrals}.
\end{enumerate}
\end{remark}

\subsection{Decomposable operators and $\boldsymbol{\ks}$-direct integrals of representations}\label{subsec:direct_integrals_of_representations}
We will now define decomposable operators, and, subsequently, a decomposable representation of a group on an $\kstext$-direct integral of a measurable family of  Banach spaces, that can (and will) be called the $\kstext$-direct integral of the fiberwise representations. Both are a natural generalization of the corresponding notion in the context of the usual direct integral of separable Hilbert spaces; see Section~\ref{Hilbert space direct integrals}.

Let $(\Omega,\nu)$ be a measure space, let $\vs$ be a vector space, and let $\{\bs_{\w}\}_{\w\in\Omega}$ be a measurable family of Banach spaces over $(\Omega,\nu,\vs)$, originating from the measurable family of semi-norms $\left\{{\norm{\,\cdot\,}}_{\w}\right\}_{\w\in\Omega}$ on $\vs$. A \emph{decomposable operator $T$ on $\{\bs_{\w}\}_{\w\in\Omega}$} is a map $\w\mapsto T_\w\in \La(\bs_\w)$ ($\w\in\Omega)$ such that, for each measurable section $s$, the section $Ts$, defined by $(Ts)(\w):=T_\w(s(\w))$, is measurable again, and such that the (possibly non-measurable) function $\w\to{\norm{T_\w}}_\w$ is $\nu$-essentially bounded. Then, for $1\leq p<\infty$, $T$ induces a bounded operator $T_{p}$ (also denoted by $\left(\int_\Omega^\oplus T_\w\,\ud\nu(\w)\right)_{\ks}$) on $\left(\int_{\Omega}^{\oplus}\bs_{\w}\,\ud\nu(\w)\right)_{\ks}$:  for $[s]_\nu\in\left(\int_{\Omega}^{\oplus}\bs_{\w}\,\ud\nu(\w)\right)_{\ks}$, we let $T_{p}[s]_\nu:=[Ts]_\nu$. If the $B_\w$ are Banach lattices and $\nu$-almost all $T_\w$ are positive operators, then $T_{p}$ is a positive operator. If $\nu$-almost all $T_\w$ are lattice homomorphisms, then $T_{p}$ is a lattice homomorphism.

Let $G$ be an abstract group. A \emph{decomposable representation $\repint$ of $G$ on $\{\bs_{\w}\}_{\w\in\Omega}$} is a family $\{\repint_\w\}_{\w\in\Omega}$, where $\repint_\w$ is a representation of $G$ on $\bs_\w$ ($\w\in\Omega$), such that, for each $g\in G$, the map $\w\to\repint_\w(g)$ is a decomposable operator on $\{\bs_{\w}\}_{\w\in\Omega}$; we denote this decomposable operator by $\repint(g)$. Then, for $1\leq p<\infty$, the map $\repint$ induces a representation $\repint_p$ of $G$ as bounded operators on $\left(\int_{\Omega}^{\oplus}\bs_{\w}\,\ud\nu(\w)\right)_{\ks}$, defined by $\repint_p(g)=(\repint(g))_p=\left(\int_\Omega^\oplus\repint_{\w}(g)\,\ud\nu(\w)\right)_{\ks}$ ($g\in G$). If the $B_\w$ are Banach lattices and $\nu$-almost all $\repint_\w$ are positive representations, then $\repint_p$ is a positive representation. We call $\repint_{p}$ the \emph{$\kstext$-direct integral of the representations $\{\repint_\w\}_{\w\in\Omega}$ with respect to $\nu$}, and we write $\repint_p=\left(\int_\Omega \repint_\w\,\ud\nu(\w)\right)_{\ks}$.

If $G$ is a topological group, it is easy to write down various conditions for the decomposable representation $\{\repint_\w\}_{\w\in\Omega}$ of $G$ on $\{\bs_{\w}\}_{\w\in\Omega}$ that are sufficient to ensure the strong continuity of $\repint_p$, together with that of all $\repint_\w$ ($\w\in\Omega$). A crude and $p$-independent one is e.g.\ that there exists a constant $M$ such that ${\norm{\repint_\w(g)}}_\w\leq\M$ for all $g\in G$ and $\w\in\Omega$, and that, for each $x\in\vs$ and $\epsilon>0$, there exists a neighbourhood $U_{x,\epsilon}$ of $e$ in $G$ such that ${\norm{\repint_\w(g)[x]_\w-[x]_\w}}_\w <\epsilon$ for all $g\in U_{x,\epsilon}$ and $\w\in\Omega$. Indeed, for each $\w\in\Omega$, this certainly implies that, for all $x\in\vs$, the map $g\mapsto\repint_\w(g)[x]_\w$ is continuous at $e$. By density, the uniform boundedness of the $\repint_\w(g)$ then implies that, for all $b\in\bs_\w$, the map
$g\to\repint_\w(g)b_\w$ is continuous at $e$; consequently, this is true at all points of $G$. Hence each $\repint_\w$ is strongly continuous. The condition also implies that, for each $p$-integrable simple section, the map $g\mapsto\repint_p(g)[s]$ is continuous at $e$. By the density statement in Lemma~\ref{lem:simple_sections_lie_dense}, the uniform boundedness of the $\repint_p(g)$ then implies that $\repint_p$ is strongly continuous.

There is a natural way to obtain a decomposable operator on $\{\bs_{\w}\}_{\w\in\Omega}$ (and, consequently, bounded operators on $\left(\int_{\Omega}^{\oplus}\bs_{\w}\,\ud\nu(\w)\right)_{\ks}$ for $1\leq p<\infty$) from one suitable linear map on the `core' space $V$, as follows. Suppose that $\widetilde T$ is a linear map on the abstract vector space $V$ with the property that there exist constants $M_{\w}$ ($\w\in\Omega$) and $M$ such ${\norm{\widetilde Tx}}_\w\leq M_{\w}{\norm{x}}_\w$ ($x\in\vs$, $\w\in\Omega$) and $M_{\w}\leq M$ for $\nu$-almost all $\w$. Then, for all $\w\in\Omega$, $\ker{\norm{\,\cdot\,}}_\w$ is $T$-invariant, hence there exists a linear map on $\vs / \ker{\norm{\,\cdot\,}}_{\w}$, denoted by $T_\w$, and given by $T_\w[x]_\w=[\widetilde T x]_\w$ ($x\in \vs$). Then ${\norm{T_\w [x]_\w}}_\w\leq M_{\w}{\norm{[x]_w}}_\w$ for all $[x]_\w\in \vs/\ker{\norm{\,\cdot\,}}_\w$. This operator extends to a bounded operator on $\bs_\w$, still denoted by $T_\w$, and then ${\norm{T_\w}}_\w\leq M$ for $\nu$-almost all $\w$.
The point is that the family $\{T_\w\}$ ($\w\in\Omega)$ automatically leaves the space $\int_{\Omega}^{\oplus}\bs_{\w}\,\ud\nu(\w)$ of measurable section invariant, so that it defines a decomposable operator $T$ on $\{\bs_{\w}\}_{\w\in\Omega}$. To see this, we first note that, if $\sect\in \int_{\Omega}^{\oplus}\bs_{\w}\,\ud\nu(\w)$ is a simple section, say $\sect(\w)=[\sum_{k=1}^{n}\ind_{A_{k}}(\w) x_{k}]_{\w}$ ($\w\in\Omega$) for some $n\in\N$, $x_{1},\ldots, x_{n}\in \vs$, and measurable subsets $A_{1},\ldots, A_{n}$ of $\Omega$, then $(Ts)(\w)=T_\w[\sum_{k=1}^{n}\ind_{A_{k}}(\w) x_{k}]_{\w}=[\sum_{k=1}^{n}\ind_{A_{k}}(\w)\widetilde T x_{k}]_{\w}$. Hence $T$ is a simple section again if $s$ is. If $s$ is a measurable section, say $s(\w)=\lim_{n\to\infty}s_n(\w)$ ($\w\in\Omega$) for simple sections $s_n$, then, as a consequence of the continuity of the $T_\w$ on $B_\w$, we see that $(Ts)(\omega)=T_\w(s(\w))=\lim_{n\to\infty}T_\w(s_n(\w))=\lim_{n\to\infty}(Ts_n)(\omega)$ ($\w\in\Omega$). Hence $Ts$ is a measurable section again if $s$ is, as desired, and the family $\{T_\w\}_{\w\in\Omega}$ is indeed a decomposable operator. We conclude that, for $1\leq p<\infty$, this `core' linear map $\widetilde T$ gives rise to a bounded operator $T_p$ on $\int_{\Omega}^{\oplus}\bs_{\w}\,\ud\nu(\w)$ such that $\norm{T_p}\leq M$.

If the $B_\w$ are Banach lattices, and $\widetilde T$ is a positive operator on $V$, then all $T_\w$ and $T_p$ are positive operators. If $\widetilde T$ is a lattice homomorphism, then all $T_\w$ and $T_p$ are lattice homomorphisms.

Consequently, there is also a natural way to obtain a decomposable representation of a group $G$ from one `core' representation $\widetilde\repint$ of $G$ on $V$. We say that $\widetilde \repint$ is \emph{pointwise essentially bounded} if, for all $g\in G$, there exist constants $M_{\w,g}$ ($\w\in\Omega$) and $M_g$ such that ${\norm{\widetilde\repint(g)x}}_\w\leq M_{\w,g}{\norm{x}}_\w$ for all $x\in\vs$ and $\w\in\Omega$, and $M_{\w,g}\leq M_g$ for $\nu$-almost all $\w$. It is immediate from the above, applied to each $\widetilde\repint(g)$ ($g\in G$), that there is a family $\{\repint_\w\}_{\w\in\Omega}$ of representations of $G$ as bounded operators on the spaces $B_\w$ that constitutes a decomposable representation $\repint$ of $G$; these are determined by $\rep_\w(g)[x]_\w=[\widetilde\repint(g)x]_\w$ ($g\in G, x\in \vs, \w\in\Omega$). Therefore, for $1\leq p<\infty$, the $\kstext$-direct integral $\repint_p=\left(\int_\Omega^\oplus\repint_\w\,\ud\nu(\w)\right)_{\ks}$ of the representations  $\{\repint_\w\}_{\w\in\Omega}$ can also be defined, and it lets $G$ act as bounded operators on $\left(\int_\Omega \bs_\w\,\ud\nu(\w)\right)_{\ks}$.
If the $B_\w$ are Banach lattices, and $\widetilde\repint$ is a positive representation of $G$ on $V$, then all $\repint_\w$ are positive representations, and hence so is $\repint_p$ ($1\leq p<\infty$).

As will become clear in Section~\ref{sec:disintegrating_space_actions}, the $\kstext$-direct integrals of positive representations that are the main concern of this paper are of the latter form. They originate from one `core' canonical positive representation of a group on one `core' vector space of simple functions on a measurable space, with $M_{g,\w}=1$ for all $g\in G$ and $\w\in\Omega$.

If, still in this context of a `core' representation, one requires crudely that there exists a constant~$M$ such that ${\norm{\widetilde\repint(g)x}}_\w\leq M{\norm{x}}_\w$ for all $g\in G$, $x\in\vs$, and $\w\in\Omega$, and that, for each $x\in\vs$ and $\epsilon>0$, there exists a neighbourhood $U_{x,\epsilon}$ of~$e$ in~$G$ such that ${\norm{\widetilde\repint (g)x-x}}_\w <\epsilon$ for all $g\in U_{x,\epsilon}$ and $\w\in\Omega$, then the family of representations $\{\repint_\w\}_{\w\in\Omega}$ satisfies the conditions as mentioned above. Therefore, in that case all representations $\repint_\w$ ($\w\in\Omega$) are strongly continuous, and so is their $\kstext$-direct integral~$\repint_p$ for $1\leq p<\infty$.

\subsection{Direct integrals of separable Hilbert spaces}\label{Hilbert space direct integrals}
In the spirit of the constant fibers in the first part of Remark~\ref{rem:Bochner}, we let $\vs$ be a (possibly complex) separable Hilbert space with norm $\norm{\,\cdot\,}$, and we take ${\norm{\,\cdot\,}}_\w=\norm{\,\cdot\,}$ for all $\w\in\Omega$. We have already seen in Remark~\ref{rem:Bochner} (this is also true for non-separable~$V$) that $\left(\int_\Omega^\oplus\bs_\w\,\ud\nu\right)_\two$ can be identified with the Bochner space $\Ell^2(\Omega,V,\nu)$. If $V$ is separable, then our $\twotext$-direct integral is also the usual Hilbert space direct integral of copies of $V$ over $\Omega$ as defined in e.g.\ \cite[p.~15--16]{Nielsen80}, and our notion of decomposable operators also coincides with the usual one as in \cite[p.~18]{Nielsen80}.

To see this, we first note that \cite[Theorem~E.9]{Cohn93} and \cite[Proposition~E.2]{Cohn93} imply that our measurable sections are precisely the Borel measurable $V$-valued functions on $\Omega$, as a consequence of the separability of $V$. Consequently, our space $\left(\int_\Omega^\oplus\bs_\w\,\ud\nu\right)_\two$\textemdash that can be supplied with an inner product in the obvious way\textemdash of (equivalence classes of) square integrable measurable sections coincides with the space of (equivalence classes of) square integrable Borel measurable $V$-valued functions, i.e.\ with the Hilbert space direct integral of copies of $V$ as in \cite[p.~15-16]{Nielsen80}.

The decomposable operators $T$ on this common space, as considered in \cite[p.~18]{Nielsen80}, are a family of bounded operators $\{T_\w\}_{\w\in\Omega}$ such that the map $\w\mapsto{\norm{T_\w}}$ is $\nu$-essentially bounded and such that, for all $x,y\in\vs$, the function $\w\mapsto(T_\w x,y)$ is Borel measurable. This notion is the same as ours. To see this, let $T$ be a decomposable operator in our sense. Then, for each $x\in V$, the image of the measurable section $\ind_\Omega x$ is a measurable section again, i.e.\ the map $\w\mapsto T_\w x$ is a measurable section for all $x\in \vs$. As already noted, this implies (and is in fact equivalent to) the Borel measurability of this $V$-valued function. Certainly the function $\w\mapsto(T_\w x,y)$ is then Borel measurable for all $y\in\vs$, i.e. the operator $T$ is decomposable in the sense of \cite[p.~18]{Nielsen80}. Conversely, suppose that $T$ is a decomposable operator in the sense of \cite[p.~18]{Nielsen80}. Then, for all $x,y\in \vs$, the function $\w\mapsto (T_\w x,y)$ is Borel measurable for all $y\in \vs$. As is easily seen, the map $\w\mapsto (T_\w s(\w),y)$ is then also Borel measurable for all simple sections $s$ and all $y\in\vs$. By the continuity of the $T_\w$, the function $\w\mapsto (T_\w s(\w),y)$ is then in fact Borel measurable for all measurable sections $s$ and all $y\in V$. By \cite[Theorems~E.9 and~E.2]{Cohn93}, this implies that the map $\w\mapsto T_w s(\w)$ is a measurable section in our sense for all measurable sections $s$. Hence $T$ is a decomposable operator in our sense.

We conclude that the theory of $\twotext$-direct integrals and their decomposable operators includes the usual one of direct integrals of copies of a separable Hilbert space and their decomposable operators. In the Hilbert space context, the next step is to piece together such direct integrals for the dimensions $1,2,\ldots,\infty$. Since this is also possible for the $\twotext$-direct integrals (see Section~\ref{subsec:perspective}), the classical theory of direct integrals of separable Hilbert spaces and their decomposable operators is included in that for the general Banach space case. 

\subsection{Perspectives in representation theory}\label{subsec:perspective}
Although we do not need this ourselves, we note that a natural further generalization of the material in Sections~\ref{subsec:direct_integrals_of_Banach_spaces} and~ \ref{subsec:direct_integrals_of_representations} is possible. First, as in \cite{HaLeRa91}, one can consider more general K\"othe spaces than $\Ell^p$-spaces, provided that the proofs of Lemma~ \ref{lem:simple_sections_lie_dense} and Proposition~\ref{prop:direct_integral_is_complete} still work, or that alternate proofs of completeness can be given that also control the measurability issue. Second, as in \cite[p.~61]{HaLeRa91}, one can work with a decomposition $\Omega=\bigsqcup_{\alpha\in A}\Omega_\alpha$ of the measure space into measurable parts. At a modest price of some extra remarks and notation, one can let the  `core' data $(\vs_\alpha,\{\w_\alpha\}_{\w_\alpha\in\Omega_\alpha})$ of a vector space $\vs_\alpha$ and a measurable family of semi-norms on $\vs_\alpha$ depend on the part $\Omega_\alpha$. If $G$ is a group, one can work with triples $(\vs_\alpha,\{\w_\alpha\}_{\w_\alpha\in\Omega_\alpha},\repint_\alpha)$, where $\repint_\alpha$ is a decomposable representation of $G$, consisting of a family of representations $\{\repint_{\w_\alpha}\}_{\w_\alpha\in\Omega_\alpha}$ of $G$ on the corresponding members of the associated family of Banach spaces $\{\bs_{\w_\alpha}\}_{\w_\alpha\in\Omega_\alpha}$, satisfying the appropriate boundedness condition. Depending on $\alpha$, this $\repint_\alpha$ may or may not originate from a common `core' representation of $G$ on $V_\alpha$. If, for each $g\in G$, there exists a constant $M_g$ such that ${\norm{\repint_{\w_\alpha} (g)x_{\w_\alpha}}}_{\w_\alpha}\leq M_g{\norm{x_{\w_\alpha}}}_{\w_\alpha}$ for all $\alpha\in A$, $\w_\alpha\in\Omega_\alpha$, and $x_{\w_\alpha}\in V_\alpha$ (this can obviously be relaxed), then the $\repint_\alpha$ yield a representation of $G$ as bounded operators on the entire direct integral of Banach spaces over $\Omega$. This representation can be viewed as the fiberwise representations $\repint_{\w_\alpha}$ ($\alpha\in A,\w_\alpha\in\Omega$) having been `glued together' via the requirement of measurability in the constructions.

 Thus the formalism provides a flexible way to construct a Banach space representation of a group (or, with obvious modifications, of an algebra) that is a direct integral of fiberwise representations on possibly different spaces. Coming from the other direction, one can ask whether a given representation of a group or algebra on a Banach space is of this form, where the fibers are to satisfy an additional condition, or are to satisfy such a condition almost everywhere. Topological irreducibility or algebraic irreducibility are natural conditions for general Banach spaces. For Banach lattices and positive representations, order indecomposability\textemdash as in this paper\textemdash is likewise natural. Theorems~\ref{thm:disintegrating_space_actions} and~\ref{thm:disintegrating_Markov_actions} shows that in certain situations a decomposition of the latter type is possible, where a one-part $\Omega$ and a decomposable representation on this single part that comes from one `core representation' on the pertinent single `core' space $\vs$ already suffice.

\section{Disintegration: action on underlying space}
\label{sec:disintegrating_space_actions}

In this section, the principal aim is Theorem~\ref{thm:disintegrating_space_actions} in Section~\ref{subsec:disintegrating_the_representation}, giving a disintegration of canonical representations as isometric lattice automorphisms on $\Ell^p$-spaces into order indecomposables. The main tool for this is the factorization Theorem~\ref{thm:factorization} for the integral over the space, as established in Section~\ref{subsec:disintegrating_the_measure}. We conclude with a worked example in Section~\ref{subsec:worked_example}.

If $\ts$ is a metric space, then we let $\PR$ be the set of Borel probability measures on $\ts$. If the group $G$ acts as Borel measurable transformations on $\ts$, then $\INV$ is the set of all $G$-invariant Borel probability measures on $\ts$, and $\ERG$ is the set of all ergodic Borel probability measures on $\ts$. Hence $\ERG\subseteq\INV\subseteq\PR$. We suppress the space and the group in the notation, as these will be clear from the context.

We recall that the canonical map from the set of Borel probability measures on a metric space $\ts$ into $\Cb(\ts)^*$, the norm dual of the bounded continuous functions on $\ts$, is injective; this follows from part of the argument to prove that (ii) implies (iv) in \cite[Theorem~17.20]{Kechris95}, combined with \cite[Theorem~17.10]{Kechris95}. We may thus view $\PR$, $\INV$, and $\ERG$ as subsets of $(\Ce_{\text{b}}(\ts))^{*}$, and we supply these sets with the induced weak$^{*}$-topologies and the ensuing Borel $\sigma$-algebras.

\subsection{Disintegrating the measure}\label{subsec:disintegrating_the_measure}

The factorization Theorem~\ref{thm:factorization} is based on a disintegration theorem for the elements of $\INV$. In order to formulate the latter, and also for future use, we start with a preliminary measurability result.

\begin{lemma}\label{lem:measurability}
Let $\ts$ be a separable metric space, and let $f:\ts\to [0,\infty]$ be a Borel measurable extended function. Then the map $\PR\rightarrow [0,\infty]$, defined by $\lambda\mapsto \int_{\ts}f(x)\,\ud\lambda(x)$, is Borel measurable.
\end{lemma}
\begin{proof}
We know from \cite[Lemma 2.3]{Varadarajan63} that the Borel $\sigma$-algebra of $\PR$ is also the smallest $\sigma$-algebra of subsets of $\PR$ such that, for all Borel subsets $Y$ of $\ts$, the map $\PR\to[0,1]$, defined by $\lambda\mapsto\lambda(Y)$, is measurable.
Thus the statement holds if $f=\ind_{Y}$ for a measurable subset $Y$ of $\ts$. By linearity it also holds for simple functions, and, using the monotone convergence theorem, it is then seen to be valid for general Borel measurable $f:\ts\to[0,\infty]$.
\end{proof}

We will now summarize what we need from the work of Farrell~\cite{Farrell62} and Varadarajan~\cite{Varadarajan63}, as it can be found in \cite[Theorem 27.5.7]{Zakrzewski02}.
Applying Lemma~\ref{lem:measurability} to $f=\ind_{Y}$ for a Borel subset $Y$ of $\ts$, we see that the integrand in part~\ref{thm:measure_disintegration_3}  of the following result is Borel measurable.

\begin{theorem}\label{thm:measure_disintegration}
Let $(G,\ts)$ be a Polish topological dynamical system, where $G$ is locally compact. Suppose that $\INV\neq\emptyset$. Then $\ERG\neq \emptyset$, and there exists a Borel measurable map $\beta:\ts\rightarrow \ERG$, $x\mapsto \beta_{x}$, with the following properties:
\begin{enumerate}
\item\label{thm:measure_disintegration_1} $\beta_{gx}=\beta_{x}$ for all $x\in X$ and $g\in G$;
\item\label{thm:measure_disintegration_2} $\lambda(\beta^{-1}(\{\lambda\}))=1$ for all $\lambda\in \ERG$;
\item\label{thm:measure_disintegration_3} For all $\mu\in \INV$ and all Borel subsets $Y$ of $\ts$,
\begin{align*}
\mu(Y)=\int_{\ts}\beta_{x}(Y)\,\ud\mu(x).
\end{align*}
\end{enumerate}
\end{theorem}

A map $\beta$ as in Theorem~\ref{thm:measure_disintegration} is called a \emph{decomposition map}.

\begin{remark}\label{rem:measure_disintegration}\quad
\begin{enumerate}
\item\label{rem:measure_disintegration_1} If $\beta$ and $\beta^\prime$ are two decomposition maps, then they agree outside a Borel subset of $\ts$ that has zero measure under all invariant Borel probability measures on $\ts$; see~\cite[proof of Lemma~4.4]{Varadarajan63}.
\item\label{rem:measure_disintegration_2} We mention for the sake of completeness that \cite[Theorem~27.5.7]{Zakrzewski02} also asserts that $\INV$ and $\ERG$ are both Borel subsets of $\PR$. Furthermore, $\PR$ is Polish; see~\cite[p.~1118]{Zakrzewski02}.
\item\label{rem:measure_disintegration_3} It is worth noting that, if, in addition, $G$ is compact, then the ergodic Borel probability measures $\ERG$ on $\ts$ are in one-to-one correspondence with the orbits of $G$ in $\ts$, as follows. For $x_0\in\ts$, one associates with the $G$-orbit $Gx_0$ the Borel measure $\lambda_{Gx_0}$ on $\ts$ by
\begin{align}\label{eq:ergodic_orbit_measure_formula}
\lambda_{Gx_0}(Y):=\mu_{G}\left(\{g\in G : gx_0\in Y\}\right),
\end{align}
where $Y$ is a Borel subset $\ts$ and $\mu_{G}$ is the normalized Haar measure on $G$; this does not depend on the choice of the point $x_0$ in the orbit.  The $\lambda_{Gx_0}$ is the unique ergodic Borel probability measure supported on $Gx_0$, and the map $Gx_0\mapsto\lambda_{Gx_0}$ is a bijection between the set of $G$-orbits and $\ERG$; see \cite[p.~1119]{Zakrzewski02}. Since $\lambda_{Gx_0}$ is simply the push-forward of $\mu_G$ to $\ts$ via the map $g\mapsto gx_0$ ($g\in G$), we then have, for every bounded Borel measurable function $f$ on $\ts$,
\begin{equation}\label{eq:ergodic_integral_formula}
\int_\ts f(x)\,\ud\mu_{Gx_0}(x)=\int_G f(g x_0)\,\ud\mu_G(g).
\end{equation}
We will use this in Section~\ref{subsec:worked_example}.
\item\label{rem:measure_disintegration_4} If $(G,\ts)$ is a topological dynamical system, where $G$ is a compact Hausdorff group and $\ts$ is a locally compact Hausdorff space, then in \cite{Seda-Wickstead76} there is a description  of the invariant Baire measures on $\ts$ that is not unsimilar to Theorem~\ref{thm:measure_disintegration}. It seems plausible that also in this context factorization and disintegration theorems analogous to Theorem~\ref{thm:factorization} and~ \ref{thm:disintegrating_space_actions} can be obtained.
\item\label{rem:measure_disintegration_5} It is known that $\INV\neq\emptyset$ if, in addition, $G$ is amenable and $\ts$ is compact; see~\cite[Theorem~5.5]{Zakrzewski02}.
\end{enumerate}
\end{remark}

We fix a decomposition map $\beta$, and proceed towards the factorization Theorem~\ref{thm:factorization}. We need the following preparatory lemma. The function $f^\prime$ that occurs in it was also introduced in \cite[Lemma~4.3]{Varadarajan63} (with a similar statement) in the case where $f$ is a bounded Borel measurable function on $\ts$, but for us it essential that $f$ need not even be finite-valued.

\begin{lemma}\label{lem:integral_disintegration_basis} Let $(G,X)$ be a Polish topological dynamical system, where $G$ is locally compact, let $\mu\in\INV$, and let $f:\ts\to [0,\infty]$ be a Borel measurable extended function on $\ts$. For $x\in\ts$, define $f^\prime(x):=\int_{\ts} f(t)\,\ud\beta_{x}(t)$. Then $f^\prime:\ts\to[0,\infty]$ is Borel measurable, and the equality
\begin{align}\label{eq:integral_disintegration_basis}
\int_{\ts}f(x)\,\ud\mu(x)=\int_{\ts}f^\prime(x)\,\ud\mu(x).
\end{align}
holds in $[0,\infty]$.
\end{lemma}

\begin{proof}
The Borel measurability of $\beta$ and Lemma~\ref{lem:measurability} imply that $f^\prime$ is Borel measurable.
For the equality of the integrals, we first suppose that $f=\ind_{Y}$ for a Borel subset $Y$ of $\ts$. Then
\begin{align*}
\int_{\ts}f(x)\,\ud\mu(x)=\mu(Y)=\int_{\ts}\beta_{x}(Y)\,\ud\mu(x)=\int_{\ts}f^\prime(x)\,\ud\mu(x)
\end{align*}
by  part~\ref{thm:measure_disintegration_3} of Theorem~\ref{thm:measure_disintegration}. By linearity, this extends to the case where $f\geq 0$ is a simple function. We now choose a sequence $(f_{k})_{k=1}^{\infty}$ of simple functions such that $0\leq f_{k}\uparrow f$ pointwise on $\ts$.  By the monotone convergence theorem, $0\leq f_{k}^\prime(x)\uparrow f^\prime(x)$ for all $x\in X$ as $k\to\infty$. Two more applications of the monotone convergence theorem, combined with what we have already shown for the $f_k$, yield
\begin{align*}
\int_{\ts}f(x)\,\ud\mu(x)=\lim_{k\to\infty}\int_{\ts}f_{k}(x)\,\ud\mu(x)=\lim_{k\to\infty}\int_{\ts}f_k^\prime(x)\,\ud\mu(x)=\int_{\ts}f^\prime(x)\,\ud\mu(x).
\end{align*}
\end{proof}

Now the proof of the factorization theorem for the integral, which is reminiscent of the familiar combination of the Tonelli and Fubini theorems, is hardly more than a formality. To this end, we let $\nu$ be the push-forward measure of $\mu$ via the Borel measurable map $\beta:\ts\to\ERG$; thus $\nu$ is the Borel probability measure on $\ERG$ given by $\nu(A):=\mu(\beta^{-1}(A))$ for a Borel subset $A$ of $\ERG$. By general principles, see~\cite[Proposition~2.6.5]{Cohn93}, if $h:\ERG\rightarrow [0,\infty]$ is a Borel measurable extended function, then the equality
\begin{align}\label{eq:push_forward_formula}
\int_{\ts}(h\circ\beta)(x)\,\ud\mu(x)=\int_{\ERG}h(\lambda)\,\ud\nu(\lambda)
\end{align}
holds in $[0,\infty]$.

\begin{theorem}\label{thm:factorization}Let $(G,\ts)$ be a Polish topological dynamical system, where $G$ is locally compact, and let $\mu\in\INV$.
\begin{enumerate}
\item\label{thm:tonelli} If $f:\ts\to[0,\infty]$ is Borel measurable, then the extended function $\lambda\mapsto\int_\ts f(x)\,\ud\lambda(x)$, with values in $[0,\infty]$, is a Borel measurable function on $\ERG$. Furthermore, the equality
\begin{equation*}\label{eq:tonelli}
\int_\ts f(x)\,\ud\mu(x)=\int_{\ERG} \left(\int_\ts f(x)\,\ud\lambda(x)\right)\,\ud\nu(\lambda)
\end{equation*}
holds in $[0,\infty]$.
\item\label{thm:fubini} If $f\in\La^1(\ts,\mu)$, then the set of $\lambda\in\ERG$ such that $f\notin \La^{1}(\ts,\lambda)$ is a Borel subset of $\ERG$ that has $\nu$-measure zero.  For $\lambda\in\ERG$, let $I_f(\lambda):=\int_\ts f(x)\,\ud\lambda(x)$ if $f\in\La^{1}(\ts,\lambda)$, and let $I_f(\lambda):=0$ if $f\notin\La^{1}(\ts,\lambda)$. Then $I_f\in\La^{1}(\ERG,\nu)$, and
\begin{equation*}\label{eq:fubini}
\int_\ts f(x)\,\ud\mu(x)=\int_{\ERG} I_f(\lambda)\,\ud\nu(\lambda).
\end{equation*}
\end{enumerate}
\end{theorem}

\begin{proof}
As to the first statement, we define $h(\lambda)=\int_\ts f(t)\,\ud\lambda(t)$. Lemma~\ref{lem:measurability} shows that $h$ is a Borel measurable function on $\ERG$. In the notation of Lemma~\ref{lem:integral_disintegration_basis}, we have $f^\prime=h\circ\beta$, so that \eqref{eq:integral_disintegration_basis} reads as $\int_X f(x)\,\ud\mu(x)=\int_X (h\circ\beta)(x)\,\ud\mu(x)$. An application of \eqref{eq:push_forward_formula} completes the proof of the first part. The second statement follows easily from an application of the first statement to the positive and negative parts of $f$.
\end{proof}

\begin{remark}\quad
\begin{enumerate}
\item
It follows from part \ref{rem:measure_disintegration_1} of Remark~\ref{rem:measure_disintegration} that $\nu$ does not depend on the choice of the decomposition map $\beta$.
\item
If $f$ is the characteristic function of a Borel subset $Y$ of $\ts$, then Theorem~\ref{thm:factorization} asserts that $\mu(Y)=\int_\ERG \lambda(Y)\,\ud\nu(\lambda)$. This formula occurs in \cite[Theorem~4.4]{Varadarajan63}. We are not aware of a reference for the general theorem as above.
\end{enumerate}
\end{remark}

In the next section, we will need the following disintegration of the $p$-norm, valid in the context of Theorem~\ref{thm:factorization}.

\begin{corollary}\label{cor:disintegration_p-norm}
Let $1\leq p<\infty$, and let $f\in \La^{p}(\ts,\mu)$. Then the set of $\lambda\in\ERG$ such that $f\notin \La^{p}(\ts,\lambda)$ is a Borel subset of $\ERG$ that has $\nu$-measure zero.  For $\lambda\in\ERG$, let $n_f(\lambda):={\norm{f}}_{\La^{p}(\ts,\lambda)}$ if $f\in\La^{p}(\ts,\lambda)$, and let $n_f(\lambda):=0$ otherwise. Then $n_f\in\La^{p}(\ERG,\nu)$, and
\begin{align}\label{eq:disintegration_p-norm}
{\norm{f}}_{\La^{p}(\ts,\mu)}={\norm{n_f}}_{\La^{p}(\ERG,\nu)}=\left(\int_{\ERG}n_f(\lambda)^{p}\,\ud\nu(\lambda)\right)^{1/p}.
\end{align}
\end{corollary}

\begin{proof}
Apply part~\ref{thm:fubini} of Theorem~\ref{thm:factorization} to ${|f|}^p$.
\end{proof}

\subsection{Disintegrating the representation}\label{subsec:disintegrating_the_representation}

Throughout this section, $(G,\ts)$ is a Polish topological dynamical system, where $G$ is locally compact, such that the set $\INV$ of $G$-invariant Borel probability measures on $\ts$ is not empty, $\mu$ is an element of $\INV$, and $\nu$ is the push-forward of $\mu$ to the ergodic Borel probability measures $\ERG$ via a decomposition map $\beta\colon\ts\to\ERG$ as in Section~\ref{subsec:disintegrating_the_measure}. We fix $1\leq p<\infty$. $G$ acts canonically on $\Ell^p(\ts,\mu)$ as isometric lattice isomorphisms, and, using the framework provided in Section~\ref{subsec:direct_integrals_of_Banach_spaces}, we will now proceed to disintegrate this representation into order indecomposables as an $\kstext$-direct integral; see Theorem~\ref{thm:disintegrating_space_actions}.

Let $\vs$ be the vector lattice of all simple scalar-valued functions on $\ts$. For each $\lambda\in\ERG$ (which will play the role of $\Omega$ in Section~\ref{subsec:direct_integrals_of_Banach_spaces}),
\begin{align*}
{\norm{f}}_\lambda:={\norm{f}}_{\La^{p}(\ts,\lambda)}\qquad(f\in \vs)
\end{align*}
defines a lattice semi-norm on $\vs$; the $p$-dependence has been suppressed in the notation for simplicity. By Corollary~\ref{cor:disintegration_p-norm}, $\lambda\mapsto {\norm{f}}_{\lambda}$ is a Borel measurable function on $\ERG$ for all $f\in \vs$. Hence, in the terminology of Section~\ref{subsec:direct_integrals_of_Banach_spaces}, $\left\{{\norm{\,\cdot\,}}_{\lambda}\right\}_{\lambda\in\ERG}$ is a measurable family of lattice semi-norms on $\vs$. For each $\lambda\in\ERG$, the completion of $\vs/\ker({\norm{\,\cdot\,}}_{\lambda})$ with respect to ${\norm{\,\cdot\,}}_{\lambda}$ is the Banach lattice $\Ell^{p}(\ts,\lambda)$, so that $\{\Ell^{p}(\ts,\lambda)\}_{\lambda\in\ERG}$ is a measurable family of Banach lattices over $(\ERG,\nu,\vs)$.

A section of $\{\Ell^p(\ts,\lambda)\}_{\lambda\in\ERG}$ is now a map $\sect:\ERG\to\bigsqcup_{\lambda\in\ERG}\Ell^p(\ts,\lambda)$ such that $\sect(\lambda)\in\Ell^p(\ts,\lambda)$ for each $\lambda\in\ERG$. A simple section is a section $\sect$ for which there exist $n\in\N$, simple functions $f_{1},\ldots, f_{n}$ on $\ts$, and Borel subsets $A_{1},\ldots, A_{n}$ of $\ERG$ such that $\sect(\lambda)=\left[\sum_{k=1}^{n}\ind_{A_{k}}(\lambda) f_{k}\right]_{\lambda}$ for all $\lambda\in\ERG$.  A section $\sect$ of $\{\Ell^{p}(\ts,\lambda)\}_{\lambda\in\ERG}$ is measurable if there exists a sequence $(\sect_{k})_{k=1}^{\infty}$ of simple sections such that ${\norm{\sect(\lambda)-\sect_{k}(\lambda)}}_{\lambda}\to 0$ as $k\to\infty$ for all $\lambda\in\ERG$.

The direct integral $\int_{\ERG}^{\oplus}\Ell^{p}(\ts,\lambda)\,\ud\nu(\lambda)$ consists of the equivalence classes $[\sect]_{\nu}$ (for the equivalence relation of $\nu$-almost everywhere equality) of measurable sections $\sect$ of $\{\Ell^{p}(\ts,\lambda)\}_{\lambda\in\ERG}$; the $\kstext$-direct integral $\left(\int_{\ERG}^{\oplus}\Ell^{p}(\ts,\lambda)\,\ud\nu(\lambda)\right)_\ks$ consists of those $[\sect]_{\nu}\in \int_{\ERG}^{\oplus}\Ell^{p}(\ts,\lambda)\,\ud\nu(\lambda)$ for which the (measurable) function $\lambda\mapsto {\norm{\sect(\lambda)}}_{\Ell^{p}(\ts,\lambda)}$ is an element of $\La^{p}(\ERG,\nu)$, and it carries the norm
\begin{align*}
{\norm{[\sect]_{\nu}}}_{p}:=\left(\int_{\ERG}{\norm{\sect(\lambda)}}_{\lambda}^{p}\,\ud\nu(\lambda)\right)^{1/p}\quad([s]_\nu\in\left(\int_{\ERG}^{\oplus}\Ell^{p}(\ts,\lambda)\,\ud\nu(\lambda)\right)_\ks).
\end{align*}
By Proposition~\ref{prop:direct_integral_is_complete}, $\left(\int_{\ERG}^{\oplus}\Ell^{p}(\ts,\lambda)\,\ud\nu(\lambda)\right)_\ks$ is a Banach lattice when supplied with this norm and with the ordering defined by
\begin{align*}
[\sect]_{\nu}\geq 0\Leftrightarrow \textrm{$\sect(\lambda)\geq 0$ in $\Ell^{p}(\ts,\lambda)$ for $\nu$-almost all $\lambda\in\ERG$}
\end{align*}
for $[\sect]_{\nu}\in\left(\int_{\ERG}^{\oplus}\Ell^{p}(\ts,\lambda)\,\ud\nu(\lambda)\right)_\ks$.

After having thus set the scene, the first thing to show that is the Banach lattices $\Ell^p(\ts,\mu)$ and $\left(\int_{\ERG}^{\oplus}\Ell^{p}(\ts,\lambda)\,\ud\nu(\lambda)\right)_\ks$ are isometrically lattice isomorphic. The basic idea for the pertinent map is quite simple: if $[f]_\mu\in\Ell^p(\ts,\mu)$ is given, this should correspond to the $\nu$-equivalence class of the section $\lambda\mapsto[f]_\lambda$ ($\lambda\in\ERG$). Apart from measurability issues, there are two problems here: $f$ need not be in $\mathcal L^p(\ts,\lambda)$ for all $\lambda$, and the image of $[f]_\mu$ could conceivably depend on the chosen representative $f$. As we will see, there exists a solution to the first problem such that the second does not occur, and such that there are no measurability issues. We make some further comments on this at the conclusion of the example in Section~\ref{subsec:worked_example}.

Implementing what will turn out to be the solution, we define, for $f\in\La^{p}(\ts,\mu)$, the section $\sect_{f}$ of $\{\Ell^{p}(\ts,\lambda)\}_{\lambda\in\ERG}$ by
\begin{align}\label{eq:specific_section}
\sect_{f}(\lambda):=\begin{cases}[f]_{\lambda} & \text{if }f\in\La^{p}(\ts,\lambda);
\\ [0]_{\lambda} & \text{otherwise}.\end{cases}
\end{align}
We know from Corollary~\ref{cor:disintegration_p-norm} that the exceptional set in this definition is a Borel subset of $\ERG$ that has $\nu$-measure zero. This easily implies that, for $f,g\in\La^{p}(\ts,\mu))$,
\begin{equation}\label{eq:additivity}
\sect_{f+g}(\lambda)=\sect_f(\lambda)+\sect_g(\lambda)\textup{ for }\nu\textup{-almost }\lambda\in\ERG,
\end{equation}
and clearly
\begin{equation}\label{eq:homogeneity}
\sect_{\alpha f}=\alpha\sect_f \quad(\alpha\in\R,\, f\in\La^{p}(\ts,\mu)).
\end{equation}

The following result takes care of measurability.

\begin{lemma}\label{lem:section_is_measurable}
Let $f\in\La^{p}(\ts,\mu)$, and  define $\sect_{f}$ as in \eqref{eq:specific_section}. Then $\sect_f$ is a measurable section of $\{\Ell^{p}(\ts,\lambda)\}_{\lambda\in\ERG}$.
\end{lemma}

\begin{proof}
There exists a sequence $(f_{k})_{k=1}^{\infty}\subseteq\vs$ such that, for all $x\in\ts$, $\abs{f_{k}(x)}\leq \abs{f(x)}$ and $f_{k}(x)\to f(x)$ as $k\to\infty$. Let $A:=\left\{\lambda\in\ERG : f\in\La^{p}(\ts,\lambda)\right\}$, and,  for $k\in\N$, let $\sect_{k}(\lambda):=[\ind_A(\lambda)f_{k}]_{\lambda} $ ($\lambda\in\ERG$). Since $A$ is a Borel subset of $\ERG$, $\sect_{k}$ is a simple section for each $k\in\N$. For $\lambda\notin A$, we have $\sect_{k}(\lambda)=[0]_{\lambda}=\sect_{f}(\lambda)$ for all $k\in\N$. For $\lambda\in A$, the dominated convergence theorem implies that
\begin{align*}
{\norm{\sect_{f}(\lambda)-\sect_{k}(\lambda)}}_{\lambda}={\norm{[f]_{\lambda}-[f_{k}]_{\lambda}}}_\lambda\to 0
\end{align*}
as $k\to\infty$.  Hence $\sect_k(\lambda)\to\sect_{f}(\lambda)$ for all $\lambda\in\ERG$, and we conclude that $\sect_{f}$ is measurable.
\end{proof}

If $f,g\in\La^{p}(\ts,\mu)$, and $[f]_\mu=[g]_\mu$, then, as the reader will easily verify, it follows from an application of \eqref{eq:disintegration_p-norm} to $f-g$ that $s_f(\lambda)=s_g(\lambda)$ for $\nu$-almost $\lambda\in\ERG$. Therefore, there is a well-defined map $\isom\colon\Ell^{p}(X,\mu)\to\int_{\ERG}^{\oplus}\Ell^{p}(\ts,\lambda)\,\ud\nu(\lambda)$, given by
\[
\isom([f]_\mu):=[\sect_f]_\nu (f\in\La^{p}(\ts,\mu)).
\]
By \eqref{eq:additivity} and \eqref{eq:homogeneity}, $\isom$ is linear.

If $f\in \Ell^{p}(\ts,\lambda)$, then, in the notation of Corollary~\ref{cor:disintegration_p-norm}, $n_f(\lambda)={\norm{\sect_f(\lambda)}}_\lambda$ for all $\lambda\in\ERG$. Since $n_f\in\La^p(\ERG,\nu)$ by Corollary~\ref{cor:disintegration_p-norm}, we have $S([f]_\mu)\in \left(\int_{\ERG}^{\oplus}\Ell^{p}(\ts,\lambda)\,\ud\nu(\lambda)\right)_\ks$. In fact,  \eqref{eq:disintegration_p-norm} yields
\begin{align*}
{\norm{[f]_\mu}}_{\Ell^p(\ts,\mu)}=\left(\int_\ERG n_f(\lambda)^p\,\ud\nu(\lambda)\right)^{1/p}=\left(\int_\ERG {\norm{\sect_f(\lambda)}}_\lambda^p\,\ud\nu(\lambda)\right)^{1/p}={\norm{\isom([f]_\mu)}}_p.
\end{align*}

We conclude that $\isom$ is an isometry of $\Ell^p(X,\mu)$ into $\left(\int_{\ERG}^{\oplus}\Ell^{p}(\ts,\lambda)\,\ud\nu(\lambda)\right)_{\ks}$.

In fact, $S$ is also surjective. To prove this, it is, according to the density statement in Lemma~\ref{lem:simple_sections_lie_dense}, sufficient to prove that all $\nu$-almost everywhere equivalence classes of simple sections are in the range of $S$. For this, in turn, it is sufficient to prove that the $\nu$-almost everywhere equivalence class of every simple section of the form $s(\lambda)=[\ind_A (\lambda)f]_\lambda$, where $A$ is a Borel subset of $\ERG$ and $f\in V$ is a simple function on $\ts$, is in the range of $S$. To this end, consider $f^\prime:=\ind_{\beta^{-1}(A)}f$. Then $f^\prime$ is a simple function on $\ts$, so $f^\prime\in\La^p(X,\mu)$. Since $f\in\La^p(X,\lambda)$ for all $\lambda\in\ERG$, the exceptional set in \eqref{eq:specific_section} is empty, and $s_{f^\prime}(\lambda)=[f^\prime]_\lambda$ for all $\lambda\in\ERG$. We claim that $s_{f^\prime}=s$, i.e.\ that $[\ind_{\beta^{-1}(A)}f]_\lambda=[\ind_A(\lambda) f]_\lambda$ for all $\lambda\in\ERG$.  For this, we use part~\ref{thm:measure_disintegration_2} of Theorem~\ref{thm:measure_disintegration}, distinguishing two cases. If $\lambda\in A$, then $\beta^{-1}(\{\lambda\})\subseteq\beta^{-1}(A)\subseteq\ts$. Since $\lambda(\beta^{-1}(\{\lambda\}))=\lambda(X)=1$, we have $\lambda((\beta^{-1}(A))^c)=0$. But then $[\ind_{\beta^{-1}(A)}f]_\lambda=[f]_\lambda$, and this equals $[\ind_A(\lambda)f]_\lambda=[1\cdot f]_\lambda$. If $\lambda\notin A$, then $\beta^{-1}(A)\cap\beta^{-1}(\{\lambda\})=\emptyset$, so that $\lambda(\beta^{-1}(A))=0$. Hence $[\ind_{\beta^{-1}(A)}f]_\lambda=[0]_\lambda$, and again this equals $[\ind_A(\lambda)f]_\lambda=[0\cdot f]_\lambda$. Thus $s_{f^\prime}=s$, as claimed, and then certainly $S([f^\prime]_\mu)=[s_{f^\prime}]_\nu=[s]_\nu$.

Furthermore, $\isom$ is a lattice homomorphism. Indeed, if $f\in\La^p(X,\mu)$ and $\lambda\in\ERG$, then $f\in\La^p(X,\lambda)$ if and only if $|f|\in\La^p(X,\lambda)$. This implies that $s_{|f|}(\lambda)=[|f|]_{\lambda}=|[f]_\lambda|$ for all $\lambda\in\ERG$. It follows form this that $|\isom([f]_\mu)|=\isom(|[f]_{\mu}|)$ for all $f\in\La^p(X,\mu)$.

We conclude that $\isom$ is an isometric lattice isomorphism between  $\Ell^p(X,\mu)$ and $\left(\int_{\ERG}^{\oplus}\Ell^{p}(\ts,\lambda)\,\ud\nu(\lambda)\right)_{\Ell^p}$.

We will now show that, under $S$, the canonical representation of $G$ on the space $\Ell^p(X,\mu)$ corresponds to the direct integral of the canonical representations of $G$ on the spaces $\Ell^p(X,\lambda)$ for $\lambda\in\ERG$. To see this fact (which almost comes for free now), we start\textemdash in the terminology of Section~\ref{subsec:direct_integrals_of_representations}\textemdash with the canonical `core' representation $\widetilde\rep$ of $G$ on the vector space $\vs$ of simple functions on $X$, defined by $(\widetilde\rep(g)f)(x):=f(g^{-1}x)$ ($g\in G,\,f\in\vs,\, x\in X$). Since ${\norm{\widetilde \rep(g)f}_\lambda=\norm{f}}_\lambda$ ($g\in G,\,f\in \vs$), this `core' representation $\widetilde\rep$ is pointwise essentially bounded. As explained in Section~\ref{subsec:direct_integrals_of_representations}, there is then a natural family $\{\rep_\lambda\}_{\lambda\in\ERG}$ of associated representations of $G$ as bounded operators on the respective completions of the spaces $(V/\ker{\norm{\,\cdot\,}}_\lambda,{\norm{\,\cdot\,}}_\lambda)$, i.e.\ on the spaces $\Ell^p(X,\lambda)$ ($\lambda\in\ERG$); these representations are determined by $\rep_\lambda(g)[f]_{\lambda}=[\widetilde\rep(g)f]_\lambda$ ($g\in G,\,f\in\vs,\,\lambda\in\ERG$). By the density of the equivalence classes of the simple functions in each $\Ell^p(X,\lambda)$ ($\lambda\in\ERG$), we see that these representations $\rep_\lambda$, as originating from $\widetilde\rep$,  are precisely the natural representations of $G$ on the spaces $\Ell^p(X,\lambda)$. As is also explained in Section~\ref{subsec:direct_integrals_of_representations}, measurability issues related to families of operators are automatically taken care of in this situation of a `core', so that the family $\{\rep_\lambda\}_{\lambda\in\ERG}$ is a decomposable representation $\rep_p=\int^\oplus_\ERG \rep_\lambda\,\ud\nu(\lambda)$ of $G$ as bounded operators on the $\kstext$-direct integral $\left(\int_\ERG^\oplus \Ell^p(X,\lambda)\,\ud\nu(\lambda)\right)_\ks$. We claim that the canonical representation $\rep_\mu$ on $\Ell^p(X,\mu)$ and the representation $\rep_p$ on $\left(\int_\ERG^\oplus \Ell^p(X,\lambda)\,\ud\nu(\lambda)\right)_\ks$ correspond under the isomorphism $\isom$ between these spaces. To see this, we let $f\in V\subseteq\La^p(X,\mu)$ and $g\in G$. Then $\widetilde\rep(g)f\in\La^p(X,\lambda)$ for all $\lambda\in\ERG$, so that $s_{\widetilde\rep(g)f}(\lambda)=[\widetilde\rep(g)f]_\lambda$ for all $\lambda\in\ERG$. Unwinding the definitions, we then see that
\begin{align*}
S(\rep_{\mu}(g)[f]_\mu)&=S([\widetilde\rep(g)f]_\mu)\\
&=[s_{\widetilde\rep(g)f}]_\nu\\
&=[\lambda\mapsto[\widetilde\rep(g)f]_\lambda]_\nu\\
&=[\lambda\mapsto\rep_\lambda(g)[f]_\lambda]_\nu\\
&=\left(\int_\ERG^\oplus \rep_\lambda(g)\,\ud\nu(\lambda)\right)_\ks\,\, ([\lambda\mapsto[f]_\lambda]_\nu)\\
&=\left(\int_\ERG^\oplus \rep_\lambda(g)\,\ud\nu(\lambda)\,\,\right)_\ks([s_f]_\nu)\\
&=\left[\left(\int_\ERG^\oplus \rep_\lambda\,\ud\nu(\lambda)\right)_\ks(g)\right]\,\, (S([f]_\mu)).
\end{align*}
By the density of the $\mu$-equivalence classes of the simple functions in $\Ell^p(X,\mu)$, our claim then follows.

We collect some of the main results so far in the following theorem. The added final part follows from part~\ref{prop:characterization_band_irreducibility_4} of Proposition~\ref{prop:characterization_band_irreducibility}, and it shows that the canonical representation of $G$ as isometric lattice automorphisms of $\Ell^p(\ts,\mu)$ can be disintegrated into order indecomposable similar representations.

\begin{theorem}\label{thm:disintegrating_space_actions}
Let $(G,\ts)$ be a Polish topological dynamical system, where $G$ is locally compact. Suppose that there exists an invariant Borel probability measure $\mu$ on $\ts$.  Let $\ERG$ be the non-empty set of ergodic Borel probability measures on $\ts$, and supply $\ERG$ with the weak$^*$-topology induced by $\Ce_{\textup{b}}(\ts)$.

Let $\beta:\ts\to\ERG$ be a decomposition map as in Theorem~\ref{thm:measure_disintegration}, and let $\nu$ be the push-forward measure of $\mu$ via $\beta$, so that $\nu$ is a Borel probability measure on $\ERG$ that is independent of the choice of $\beta$.

Let $1\leq p<\infty$.

\begin{enumerate}
\item Let $V$ be the vector space of simple functions on $\ts$. Then ${\{{\norm{\,\cdot\,}}_\lambda\}}_{\lambda\in\ERG}$ is a measurable family of semi-norms on $V$. The resulting family of completions of the spaces $(V/\ker{\norm{\,\cdot\,}}_\lambda,{\norm{\,\cdot\,}}_\lambda)$ is the family $\{\Ell^p(\ts,\lambda)\}_{\lambda\in\ERG}$, which is a measurable family of Banach lattices over $(\ERG,\nu,V)$. Therefore, the $\kstext$-direct integral $\left(\int_{\ERG}^{\oplus}\Ell^{p}(\ts,\lambda)\,\ud\nu(\lambda)\right)_\ks$ of this family with respect to $\nu$ can be defined, and this space is a Banach lattice;
\item Define $\isom:\Ell^{p}(X,\mu)\to\left(\int_{\ERG}^{\oplus}\Ell^{p}(\ts,\lambda)\,\ud\nu(\lambda)\right)_{\ks}$ by $\isom([f]_{\mu}):=[\sect_{f}]_{\nu}$ \textup{(}$f\in\La^{p}(\ts,\mu)$\textup{)}, where $\sect_f$ is as defined in \eqref{eq:specific_section}. Then $\isom$ is an isometric lattice isomorphism between the Banach lattices $\Ell^{p}(X,\mu)$ and $\left(\int_{\ERG}^{\oplus}\Ell^{p}(\ts,\lambda)\,\ud\nu(\lambda)\right)_{\ks}$;
\item $S$ is an intertwining operator between the canonical  representation $\rep_\mu$ of $G$ as isometric lattice automorphisms of $\Ell^p(X,\mu)$ and the representation $\left(\int_{\ERG}^{\oplus}\!\rep_{\lambda}\,\ud\nu(\lambda)\right)_\ks$ on $\left(\int_{\ERG}^{\oplus}\Ell^{p}(\ts,\lambda)\,\ud\nu(\lambda)\right)_{\ks}$, which is the $\kstext$-direct integral of the family $\{\rep_\lambda\}_{\lambda\in\ERG}$ of canonical representations of $G$ as isometric lattice automorphisms on the Banach lattices $\Ell^p(X,\lambda)$ \textup{(}$\lambda\in\ERG$\textup{)}. That is, for each $g\in G$, the following diagram commutes:
\begin{align*}
 \xymatrixcolsep{3.5pc}\xymatrix{
 \Ellp(\ts,\mu) \ar[r]^{\rep_{\mu}(g)} \ar[d]_{\isom} &   \Ellp(\ts,\mu) \ar[d]_{\isom}\\ \left(\int_{\ERG}^{\oplus}\Ell^{p}(\ts,\lambda)\,\ud\nu(\lambda)\right)_{\ks}\quad \ar[r]^{\left(\int_{\ERG}^{\oplus}\!\rep_{\lambda}(g)\,\ud\nu(\lambda)\right)_\ks} & \quad \left(\int_{\ERG}^{\oplus}\Ell^{p}(\ts,\lambda)\,\ud\nu(\lambda)\right)_{\ks}
 }
\end{align*}
\item For all $\lambda\in\ERG$, the representation $\rep_\lambda$ of $G$ on the fiber $\Ell^p(X,\lambda)$ is order indecomposable.
\end{enumerate}
\end{theorem}

\begin{remark}\label{rem:strong_continuity_of_spaces}\quad
The strong continuity of the $\kstext$-direct integral of representations was briefly addressed in Section~\ref{subsec:group_actions}. Although strong continuity played no role in the proofs, let us still mention that in the present context this is automatic: according to Corollary~\ref{cor:automatic_strong_continuity_Polish_pair}, $\rep_\mu$
(and hence $\left(\int_\ERG^\oplus \rep_\lambda\,\ud\nu(\lambda)\right)_\ks$) and all $\rep_\lambda$ ($\lambda\in\ERG$) are strongly continuous representations.
\end{remark}

\subsection{Worked example}\label{subsec:worked_example}

We conclude this section with a simple example of a representation that we disintegrate explicitly.

Let $\Di:=\left\{z\in\C :  \abs{z}\leq 1\right\}$ and let $\T:=\left\{z\in\C : \abs{z}=1\right\}$. Then $(\T,\Di)$ is a Polish topological dynamical system with compact group when supplied with the rotation action: $(z_{1},z_{2})\mapsto z_{1}z_{2}$ ($z_{1}\in\T$, $z_{2}\in\Di$). We let $\mu$ be the normalized restriction of the Lebesgue measure on $\R^2$ to the Borel $\sigma$-algebra of $\Di$. Then $\mu$ is a $\T$-invariant Borel probability measure on $\ts$. We fix $1\leq p<\infty$. Our aim is to exhibit an explicit disintegration of $\Ell^p(\Di,\mu)$ and the representation of $\rep_\mu$ of $\T$ on this space, as provided in abstracto by Theorem~\ref{thm:disintegrating_space_actions}.

The first step is to determine the set $\ERG$ of ergodic Borel probability measures on $\Di$. We know from part 3 of Remark~\ref{rem:measure_disintegration} that these measures   are in one-to-one correspondence with the orbits of $\T$, i.e.\ with the elements of the interval $[0,1]$ that parameterizes the radius of the orbits. From \eqref{eq:ergodic_orbit_measure_formula}  we infer an explicit formula for the ergodic Borel probability measure $\lambda_r$ corresponding to an orbit of radius $r\in[0,1]$, namely
\begin{align}\label{eq:orbit_measure_formula_disc}
\lambda_{r}(Y)=\frac{1}{2\pi}\int_{[0,2\pi]}\ind_{Y}(r\ue^{\ui\theta})\,\ud\theta,
\end{align}
where $Y$ is a Borel subset of $\Di$. More generally, if $f:\Di\to\R$ is a bounded Borel measurable function, then \eqref{eq:ergodic_integral_formula} gives, for $r\in [0,1]$,
\begin{equation}\label{eq:orbit_integral_formula_disc}
\int_\Di f(z)\,\ud\lambda_r(z)=\frac{1}{2\pi}\int_{[0,2\pi]}f(r\ue^{\ui\theta})\,\ud\theta.
\end{equation}

The second step is to determine $\ERG$ as a topological space; recall that $\ERG$ is endowed with the weak$^{*}$-topology via the inclusion $\ERG\subseteq \Ce_{\mathrm{b}}(\Di)^{*}$. We know that $\ph:[0,1]\rightarrow \ERG$, given by $\ph(r)=\lambda_{r}$, is a bijection; we claim that it is even a homeomorphism. To see this, let $(r_{n})_{n\in\N}\subseteq[0,1]$ and let $r_{n}\rightarrow r\in[0,1]$ as $n\rightarrow \infty$. If $f\in\Ce_{\mathrm{b}}(\Di)$, then, using \eqref{eq:orbit_integral_formula_disc} and the dominated convergence theorem, we see that
\begin{align*}
\int_{\Di}f(z)\,\ud\lambda_{r_{n}}(z)=\frac{1}{2\pi}\int_{[0,2\pi]}f(r_{n}\ue^{\ui\theta})\,\ud\theta\rightarrow \frac{1}{2\pi}\int_{[0,2\pi]} f(r\ue^{\ui\theta})\,\ud\theta=\int_{\Di}f(z)\,\ud\lambda_{r}(z)
\end{align*}
as $n\to\infty$. Hence $\ph$ is continuous. Since $[0,1]$ is compact and $\ERG$ is Hausdorff, we conclude that $\ph$ is a homeomorphism.

The third step is to find a decomposition map $\beta:\Di\to\ERG$. In this case (as for more general actions of compact groups, where the ergodic Borel probability measures are supported on orbits), this map is already uniquely determined by parts~\ref{thm:measure_disintegration_1} and~\ref{thm:measure_disintegration_2} of Theorem~\ref{thm:measure_disintegration}.  Indeed, let $r\in[0,1]$. Then part~\ref{thm:measure_disintegration_2} shows that $\beta^{-1}(\{\lambda_r\})$ cannot be disjoint from the orbit $\T r$ (which is the support of $\lambda_r$), and subsequently part~\ref{thm:measure_disintegration_1} implies that this set contains the entire orbit. Since  $\beta^{-1}(\{\lambda_{r_1}\})$ and $\beta^{-1}(\{\lambda_{r_2}\})$ are obviously disjoint for $r_1\neq r_2$, we must have $\beta^{-1}(\{\lambda_r\}=\T r$ for all $r\in[0,1]$. We conclude that $\beta_{r\ue^{i\theta}}=\lambda_{r}$ for $r\in[0,1]$ and $\theta\in\R$ (so that $\beta$ is, in fact, uniquely determined here).

We know a priori from Theorem~\ref{thm:measure_disintegration} that $\beta$ is Borel measurable, but this can also be seen directly. In fact, $\beta$ is even continuous, because $\ph^{-1}\circ\beta:\Di\rightarrow [0,1]$ is continuous (it maps $r\ue^{\ui\theta}$ to $r$), and hence so is $\beta=\ph\circ(\ph^{-1}\circ\beta)$.

We also know a priori that part~\ref{thm:measure_disintegration_3} of Theorem~\ref{thm:measure_disintegration} is satisfied for our $\mu$, but using \eqref{eq:orbit_measure_formula_disc} this can also be seen directly.  Indeed, using polar coordinates we have, for a Borel subset $Y$ of $\Di$,
\begin{align*}
\mu(Y)&=\frac{1}{\pi}\int_{[0,1]}\int_{[0,2\pi]} r\mathbf{1}_{Y}(r\ue^{\ui\theta})\,\ud\theta\,\ud r=2\int_{[0,1]}r\lambda_{r}(Y)\,\ud r\\
&=\frac{1}{\pi}\int_{[0,1]}\int_{[0,2\pi]}r\beta_{r\ue^{\ui\theta}}(Y)\,\ud \theta\,\ud r=\int_{\Di}\beta_{z}(Y)\,\ud\mu(z).
\end{align*}

Theorem~\ref{thm:disintegrating_space_actions} gives a disintegration of the action of $\T$ on $\Ell^p(\Di,\mu)$ as an $\kstext$-direct integral of representations with $\ERG$ as underlying point set,  but it is more intuitive to formulate this with $[0,1]$, which is homeomorphic to $\ERG$, as underlying point set. Therefore, we let $\nu$ be the push-forward measure of $\mu$ via $\ph^{-1}\circ\beta\colon\Di\to [0,1]$. Thus, if $A$ is a Borel subset of $[0,1]$, then
$\nu(A)=\mu(\beta^{-1}\circ\ph(A))=\mu(\{z\in\C : |z|\in A\})$. Using polar coordinates, we see that
\[
\nu(A)=\mu(\{r\ue^{\ui\theta}: r\in A\})=\frac{1}{\pi}\int_{[0,1]}\int_{[0,2\pi]}r\ind_A(|r\ue^{\ui\theta}|)\,\ud\theta\,\ud r=\int_{[0,1]}\ind_A\cdot 2r\,\ud r.
\]
We conclude that $\nu$ is the measure $2r\,\ud r$ on the Borel subsets of $[0,1]$. For (say) a bounded Borel measurable function $f$ on $\Di$, part~\ref{thm:fubini} of the factorization Theorem~\ref{thm:factorization} then takes the form
\begin{equation}\label{eq:fubini_disc}
\frac{1}{\pi}\int_\Di f(z)\,\ud\mu(z)=\int_{[0,1]}\left(\frac{1}{2\pi}\int_{[0,2\pi]} f(r\ue^{\ui\theta})\,\ud\theta)\right)\,2r\,\ud r,
\end{equation}
where \eqref{eq:orbit_integral_formula_disc} has been used. The validity of this formula in itself is, of course, clear; the point is its interpretation as an instance of the factorization in Theorem~\ref{thm:factorization}.

Let $\vs$ be the vector lattice of simple functions on $\Di$. According to Theorem~\ref{thm:disintegrating_space_actions}, ${\{{\norm{\,\cdot\,}}_{\lambda_r}\}}_{r\in [0,1]}$ is a measurable family of semi-norms on $V$, so that ${\{\Ell^{p}(\Di,\lambda_r)\}}_{r\in[0,1]}$ is a measurable family of Banach lattices over $([0,1],\nu,\vs)$, and the $\kstext$-direct integral $\left(\int_{[0,1]}^{\oplus}\Ell^{p}(\Di,\lambda_r)\,2r\ud r\right)_\ks$ can be defined. Let $\isom:\Ell^{p}(\Di,\mu)\to \left(\int_{[0,1]}^{\oplus}\Ell^{p}(\Di,\lambda_{r})\,2r\,\ud r\right)_{\ts}$ be such that $\isom([f]_{\mu}):=[\sect_{f}]_{2r\ud r}$, where $\sect_{f}(r)=[f]_{\lambda_{r}}$ if $f\in\La^p(\Di,\lambda_r)$, and $\sect_{f}(r)=[0]_{\lambda_{r}}$ if $f\notin\La^p(\Di,\lambda_r)$. The latter  exceptional set is Borel measurable and has $2r\,\ud r$-measure zero. Equivalently, it has $\ud r$-measure zero (likewise, we could have written $[s_f]_{\ud r}$ for  $[s_f]_{2r\ud r}$). Then, according to Theorem~\ref{thm:disintegrating_space_actions}, $\isom$ is a well-defined isometric lattice isomorphism between $\Ell^{p}(\Di,\mu)$ and $\left(\int_{[0,1]}^{\oplus}\Ell^{p}(\Di,\lambda_{r})\,2r\,\ud r\right)_{\ks}$. If (for the ease of formulation) $f$ is a bounded Borel measurable function on $\Di$, then the exceptional set is empty, and, using \eqref{eq:orbit_integral_formula_disc}, the isometric nature of $S$ at the point $[f]_\mu\in\Ell^p(\ts,\mu)$ is an application of \eqref{eq:fubini_disc} to ${|f|}^p$:
\begin{align*}
{\norm{[f]}}^p_\mu &=\frac{1}{\pi}\int_{\Di}{|f(z)|}^p\,\ud\mu(z)=\int_{[0,1]}\left(\frac{1}{2\pi}\int_{[0,2\pi]} {|f(r\ue^{\ui\theta})|}^p\,\ud\theta)\right)\,2r\,\ud r\\
&=\int_{[0,1]}\left(\int_{\Di}{|f(z)|}^p\,\ud\lambda_r(z)\right)\,2r\,dr=\int_{[0,1]} {\norm{s_f(r)}}^p_{\lambda_r}\,2r\,\ud r\\
&={\norm{S([f]_\mu)}}^p_p.
\end{align*}
Furthermore, $S$ is an intertwining operator between the canonical representation $\rep_{\mu}$ of $\T$ on $\Ell^{p}(\Di,\mu)$ and the $\kstext$-direct integral $\left(\int_{[0,1]}^{\oplus}\rep_{\lambda_r}\,2r\,\ud r \right)_\ks$  of the order indecomposable representations $\rep_{\lambda_r}$ of $\T$ on the spaces $\Ell^p(\Di,\lambda_r)$. That is,  for each $z\in\T$, the diagram
\begin{align}\label{eq:diagram_disc}
 \xymatrixcolsep{4pc}\xymatrix{
 \Ellp(\Di,\mu) \ar[r]^{\rep_{\mu}(z)} \ar[d]_{\isom} &   \Ellp(\Di,\mu) \ar[d]_{\isom}\\ \left(\int_{[0,1]}^{\oplus}\Ell^{p}(\Di,\lambda_{r})\,2r\,\ud r\right)_{\ks}\quad \ar[r]^{\left(\int_{[0,1]}^{\oplus}\!\rep_{\lambda_{r}}(z)\,2r\,\ud r\right)_\ks} & \quad \left(\int_{[0,1]}^{\oplus}\Ell^{p}(\Di,\lambda_{r})\,2r\,\ud r\right)_{\ks}
 }
\end{align}
is commutative.

Intuitively, this is certainly plausible, since `restricting a function to an orbit' is clearly a $\T$-equivariant operation, and the commutativity of diagram~\eqref{eq:diagram_disc} merely reflects that this is what the operator $S$ tries to do. We write `tries to do', and not `does', because `restricting'  is meaningless for the elements of the actual domain of $S$, which are $\mu$-equivalence classes of measurable functions.  The `actual' intertwining statement in Theorem~\ref{thm:disintegrating_space_actions} is, therefore, that this intuitive observation can be modified into a form that is meaningful and that survives during the measure-theoretical constructions. In this case, it comes down to the following.

If there is  an empty exceptional set in the definition of $s_f$ for $f\in\mathcal L^p(\Di,\mu)$ (e.g.\ if $f$ is a bounded Borel measurable function), then, for each fixed $r\in [0,1]$, the value $s_f(r)=[f]_{\lambda_r}$ is clearly determined by the restriction of $f$ to the corresponding orbit of radius $r$. Since the characteristic function of this orbit is $\mu$-almost everywhere zero, it is likewise clear that $[f]_{\lambda_r}$ \emph{always} depends on the choice of the representative $f$ of $[f]_\mu$. Nevertheless, the $2r\ud r$-equivalence class of the section $r\mapsto s_f(r)=[f]_{\lambda_r}$ does \emph{not} depend on this choice. Moreover,  the map $S$ sending $[f]_{\mu}$ to this $2r\ud r$-equivalence class is (clearly) $\T$-equivariant.

Furthermore, this can still be made to work  when there \emph{is} a non-empty exceptional set in the definition of $s_f$; i.e.\ when $p$-integrability of $f$ is lost when $f$ is restricted to certain orbits. For each fixed orbit, there are evidently $f\in\mathcal L^p(\ts,\mu)$ for which this is the case, but for each fixed $f\in\mathcal L^p(\ts,\mu)$ there are $2r\ud r$-almost none of such orbits.

\section{Disintegration: general case}\label{sec:disintegrating_Markov_actions}

In Section~\ref{sec:disintegrating_space_actions}, we started with a topological dynamical system $(G,K)$ and a $G$-invariant Borel probability measure on $\ts$. In that context, there existed canonically associated strongly continuous representations of $G$ as isometric lattice automorphisms of the spaces $\Ell^p(\ts,\mu)$ ($1\leq p<\infty$) that fix the constants.

In the current section, we turn the tables.  We start with an (at first) abstract group $G$ and a probability space $(\ts,\mu)$, and we assume that, for some $1\leq p<\infty$, $G$ acts as isometric lattice automorphisms of $\Ell^p(\ts,\mu)$ such that the constants are left fixed. It is then established that, in fact, $G$ acts naturally in a similar manner on $\Ell^p(\ts,\mu)$ for \emph{all} $1\leq p<\infty$; see~Corollary~\ref{cor:family}. Furthermore, if $G$ is a locally compact Hausdorff group and the original representation is strongly continuous, then it is shown that there is an isomorphic model in which this $G$-action on all $\Ell^p$-spaces originates canonically from a measure preserving continuous $G$-action on a compact Hausdorff space; see Theorem~\ref{thm:transfer}. Under mild additional assumptions, we can then conclude from our disintegration Theorem~\ref{thm:disintegrating_space_actions} that, even though there was originally no action of $G$ on an underlying point set, the original representation(s) of $G$ on $\Ellp(\ts,\mu)$ can still be disintegrated into order indecomposable representations as an $\kstext$-direct integral. This leads to Theorem~\ref{thm:disintegrating_Markov_actions}, which is an ordered relative of the general unitary disintegration in \cite[Theorem~18.7.6]{Dixmier77}

\begin{remark}
It follows from the combination of \cite[Vol. I, Exercise~1.12.102]{Bogachev07},  \cite[Vol. II, Example~6.5.2]{Bogachev07}, and \cite[Vol. I, Exercise~4.7.63]{Bogachev07} that, for $1\leq p<\infty$, $\Ell^p(\ts,\mu)$ is always separable whenever $\ts$ is a separable metric space and $\mu$ is a Borel probability measure on $\ts$. Therefore, the representation spaces in Section~\ref{sec:disintegrating_space_actions} are all separable.  Furthermore, we have observed in Remark~\ref{rem:strong_continuity_of_spaces} that the representations on the relevant spaces in Section~\ref{sec:disintegrating_space_actions} are all strongly continuous.  Neither of these properties has played a role in the proofs so far. Quite to the contrary, in the current section both properties will be essential in order to be able to exhibit a model in Theorem~\ref{thm:transfer} to which the disintegration Theorem~\ref{thm:disintegrating_Markov_actions} can subsequently be applied.
\end{remark}

\begin{remark}
With the exception of Remark~\ref{rem:lamperti}, the combination of ideas, arguments and results in Lemma~\ref{lem:powers_preserved} up to and including Theorem~\ref{thm:transfer} is rather similar to that in \cite{EiFaHaNa15}. Unfortunately, we cannot directly apply results such as \cite[Proposition~13.6 and Theorem~13.9]{EiFaHaNa15}. The reason is that the so-called Markov operators on $\Ell^p(\ts,\mu)$ that are considered in \cite{EiFaHaNa15} are positive operators $\lh$ that fix the constants and satisfy $\int_\ts f\,\ud\mu=\int_\ts f\,\ud\mu$ for all $f\in\Ell^p(\ts,\mu)$. Our point of departure, where $\lh$ preserves the norm rather than the integral, and is a lattice homomorphism rather than merely a positive operator, is different. This necessitates an independent, albeit similar, development; see also Remark~\ref{rem:markov_typering}.
\end{remark}

We begin by showing that representations of an abstract group $G$ as isometric lattice automorphisms that fix the constants come in families.

There will be only one measure in this section, and we happily resort to the usual practice of ignoring the distinction between equivalence classes of functions and their representatives.

We start with the following key observation, see \cite[Theorem~7.23.vi]{EiFaHaNa15}.

\begin{lemma}\label{lem:powers_preserved}
Let $(\ts,\mu)$ be a probability space, and let $\lh\colon \Ell^\infty(\ts,\mu)\to \Ell^\infty(\ts,\mu)$ be a lattice homomorphism that fixes the constants. Then $\lh({|f|}^p)={|\lh(f)|}^p$ for all $f\in \Ell^\infty(\ts,\mu)$ and all $1\leq p<\infty$.
\end{lemma}

\begin{lemma}\label{lem:basic_isometry}
Let $(\ts,\mu)$ be a probability space, and let $\lh\colon \Ell^\infty(\ts,\mu)\to \Ell^\infty(\ts,\mu)$ be a lattice homomorphism that fixes the constants. Then the following are equivalent:
\begin{enumerate}
\item\label{lem:basic_isometry_1} $\int_\ts \lh (f)\,\ud\mu =\int_\ts f\,\ud\mu$ for all $f\in \Ell^\infty(\ts,\mu)$;
\item\label{lem:basic_isometry_2} ${\norm{\lh (f)}}_1={\norm{f}}_1$ for all $f\in \Ell^\infty(\ts,\mu)$;
\item\label{lem:basic_isometry_3} There exists $1\leq p<\infty$ such that ${\norm{\lh (f)}}_p={\norm{f}}_p$ for all $f\in \Ell^\infty(\ts,\mu)$;
\item\label{lem:basic_isometry_4} For all $1\leq p<\infty$,  we have ${\norm{\lh (f)}}_p={\norm{f}}_p$ for all $f\in \Ell^\infty(\ts,\mu)$.
\end{enumerate}
\end{lemma}

\begin{proof}
To see that \eqref{lem:basic_isometry_1} implies \eqref{lem:basic_isometry_4}, we use Lemma~\ref{lem:powers_preserved} to note that ${\norm{\lh(f)}}_p^p=\int_\ts {|\lh(f)|}^p\,\ud\mu=\int_\ts \lh({|f|}^p)\,\ud\mu=\int_\ts {|f|}^p\,\ud\mu={\norm{f}}_p^p$.

It is clear that \eqref{lem:basic_isometry_4} implies \eqref{lem:basic_isometry_3}.

To see that \eqref{lem:basic_isometry_3} implies \eqref{lem:basic_isometry_2}, we invoke Lemma~\ref{lem:powers_preserved} to compute as follows:
\begin{align*}
{\norm{\lh(f)}}_1&=\int_\ts |\lh(f)|\,\ud\mu = \int_\ts \lh(|f|)\,\ud\mu=\int _\ts\lh\left({\left|\left({|f|}^{1/p}\right)\right|}^p\right)\,\ud\mu\\
&=\int_\ts {\left|\lh\left({|f|}^{1/p}\right)\right|}^p\,\ud\mu={\norm{\lh\left({|f|}^{1/p}\right)}}_p^p={\norm{{|f|}^{1/p}}}_p^p\\
&={\norm{f}}_1.
\end{align*}

To see that \eqref{lem:basic_isometry_2} implies \eqref{lem:basic_isometry_1}, we note that the equality in \eqref{lem:basic_isometry_1} is just the one in \eqref{lem:basic_isometry_2} if $f\geq 0$ (note that $\lh f\geq 0$ then). For general $f$, we then have
\begin{align*}
\int_\ts \lh(f)\,\ud\mu &=\int_\ts (\lh(f))^+\,\ud\mu - \int_\ts (\lh(f))^-\,\ud\mu
=\int_\ts \lh(f^+)\,\ud\mu - \int_\ts \lh(f^-)\,\ud\mu\\
&=\int_\ts f^+ \,\ud\mu - \int_\ts f^-\,\ud\mu =\int_\ts f\,\ud\mu.
\end{align*}
\end{proof}

Fix $1\leq p<\infty$, and consider a lattice homomorphism $\lh\colon\Ell^p(\ts,\mu)\to\Ell^p(\ts,\mu)$ that leaves the constants fixed. Then $\lh$, being a positive operator on a Banach lattice, is continuous in the $p$-norm. Furthermore, $\lh$ leaves $\Ell^\infty(\ts,\mu)$ invariant. Indeed, if $f\in\Ell^\infty(\ts,\mu)$, then $|f|\leq {\norm{f}}_\infty\ind_\ts$ in the lattice $\Ell^p(\ts,\mu)$. An application of $\lh$ shows that $\lh(f)$ is in $\Ell^\infty(\ts)$ again, and also that $\lh\colon\Ell^\infty(\ts,\mu)\to\Ell^\infty(\ts,\mu)$ is contractive in the supremum-norm. For later use, let us note that the latter implies that a group of lattice automorphisms of $\Ell^p(X,\mu)$ that fixes the constants automatically acts on $\Ell^\infty(\ts)$ as isometric lattice automorphisms.

Using continuity and density arguments, the following result is now an easy consequence of Lemma~\ref{lem:basic_isometry}.

\begin{lemma}\label{lem:markov_typering}
Let $(X,\mu)$ be a probability space, let $1\leq p<\infty$, and let $\lh\colon\Ell^p(\ts,\mu)\to\Ell^p(\ts,\mu)$ be a lattice homomorphism that leaves $\ind_X$ fixed. Then $\lh$ leaves $\Ell^\infty(\ts,\mu)$ invariant, and the restriction of $\lh$ to $\Ell^\infty(\ts,\mu)$ is a contractive lattice homomorphism for the supremum-norm that leaves the constants fixed. Furthermore, the following are equivalent:
\begin{enumerate}
\item\label{lem:markov_typering_1} $\int_\ts \lh (f)\,\ud\mu =\int_\ts f\,\ud\mu$ for all $f\in \Ell^\infty(\ts,\mu)$;
\item\label{lem:markov_typering_2} $\int_\ts\lh (f)\,\ud\mu =\int_\ts f\,\ud\mu$ for all $f\in \Ell^p(\ts,\mu)$;
\item\label{lem:markov_typering_3} ${\norm{\lh (f)}}_p={\norm{f}}_p$ for all $f\in\Ell^p(\mu,\ts)$.
\end{enumerate}
\end{lemma}

\begin{remark}\label{rem:markov_typering}
In the terminology of \cite[Section~13.1]{EiFaHaNa15}, the equivalence of \eqref{lem:markov_typering_2} and \eqref{lem:markov_typering_3} in Lemma~\ref{lem:markov_typering} implies that a lattice homomorphism $\lh\colon\Ell^p(\ts,\mu)\to\Ell^p(\ts,\mu)$ that leaves the constants fixed is a Markov operator on $\Ell^p(\ts,\mu)$ precisely when it is an isometry.
\end{remark}

Note that $p$ is absent from part~\ref{lem:markov_typering_1} of Lemma~\ref{lem:markov_typering}, but present in parts~\ref{lem:markov_typering_2} and~\ref{lem:markov_typering_3}. Lemma~\ref{lem:basic_isometry} has similar features. Using restriction to, and extension from, the common dense subspace $\Ell^\infty(\ts,\mu)$ of all spaces $\Ell^p(\ts,\mu)$ ($1\leq p<\infty$), one readily obtains the following result.

\begin{lemma}\label{lem:restrictie_isomorfisme_single_operator}
Let $(X,\mu)$ be a probability space, let  $1\leq p\leq q<\infty$, and let $\lh\colon\Ell^p(\ts,\mu)\to\Ell^p(\ts,\mu)$ be an isometric lattice homomorphism that leaves the constants fixed. Then $\lh$ leaves $\Ell^q(\ts,\mu)\subseteq\Ell^p(\ts,\mu)$ invariant, and the restriction $\lh\colon\Ell^q(\ts,\mu)\to \Ell^q(\ts,\mu)$ is an isometric lattice homomorphism that leaves the constants fixed. Moreover, every isometric lattice homomorphism of $\Ell^q(\ts,\mu)$ that leaves the constants fixed can thus be obtained from a unique $\lh$.
\end{lemma}

\begin{remark}\label{rem:lamperti}
There is an alternative way to understand why Lemma~\ref{lem:restrictie_isomorfisme_single_operator} holds. According to Lamperti's theorem \cite[Theorem~3.2.5]{Fleming-Jamison03}, the isometries of $\Ell^p(\ts,\mu)$ are, for $1\leq p\neq 2<\infty$ , the composition of a multiplication operator and an operator that is induced by a regular set isomorphism. An inspection of the proof shows that the theorem actually describes all disjointness preserving isometries; this disjointness preserving property being automatic if $p\neq 2$. Consequently, if $1\leq p<\infty$ is fixed, and $\lh\colon\Ell^p(\ts,\mu)\to \Ell^p(\ts,\mu)$ is an isometric lattice isomorphism that fixes the constants, then the description in Lamperti's theorem applies to the operator $T$. Since $T$ fixes the constants, the multiplication operator is the identity, so that $T$ is actually induced by a regular set isomorphism. Since $T$ is an isometry, this regular set isomorphism must be measure preserving. It is then clear why and how $T$ acts as isometric lattice automorphisms on all spaces $\Ell^p(\ts,\mu)$: all these actions arise from the same underlying measure preserving regular set isomorphism. At the cost of invoking Lamperti's result, and of some technical details of a different nature, a different proof of Lemma~\ref{lem:restrictie_isomorfisme_single_operator} can thus be given.
\end{remark}

If $(\ts,\mu)$ is a probability space, and if $1\leq p<\infty$, then, as is well known, the topology that is induced on $\{f\in\Ell^\infty(\ts,\mu):{\norm{f}}_\infty\leq 1\}$ by $\Ell^p(\ts,\mu)$ does not depend on $p$. As a first consequence, the spaces $\Ell^p(\ts,\mu)$ for $1\leq p<\infty$ are either all separable, or all non-separable; their separability is known to be equivalent to the separability of $\mu$, see \cite[Vol. I, Exercise~4.7.63]{Bogachev07}. As a second consequence, when combined with Lemma~\ref{lem:restrictie_isomorfisme_single_operator} and with the already observed fact that a lattice homomorphism of $\Ell^p(\ts,\mu)$, that leaves the constants fixed, automatically leaves $\{f\in\Ell^\infty(\ts,\mu):{\norm{f}}_\infty\leq 1\}$ invariant, this $p$-independence of the topology yields the statement on the strong operator topology of the following result, which is in the spirit of \cite[Proposition~13.6]{EiFaHaNa15}.

\begin{proposition}\label{prop:restrictie_isomorfisme_semigroup} Let $(\ts,\mu)$ be a probability space, and let $1\leq p\leq q<\infty$. Then the semigroup/group of isometric lattice homomorphisms/automorphisms of $\Ell^p(\ts,\mu)$ into/onto itself that leaves the constants fixed is, via the restriction map, isomorphic to the semigroup/group of isometric lattice homomorphisms/automorphisms of $\Ell^q(\ts,\mu)$ into/onto itself that leaves the constants fixed. This isomorphism is a homeomorphism for both the strong and the weak operator topologies as induced from the bounded operators on $\Ell^p(\ts,\mu)$ and $\Ell^q(\ts,\mu)$.
\end{proposition}

We thus have the following result concerning our type of representations always occurring in families.

\begin{corollary}\label{cor:family}
Let $(\ts,\mu)$ be a probability space, let $G$ be a group, and let $1\leq p_0<\infty$. Suppose that $G$ acts on $\Ell^{p_0}(\ts,\mu)$ as isometric lattice automorphisms that leave the constants fixed. Then $G$ acts naturally on $\Ell^{p}(\ts,\mu)$ as isometric lattice automorphisms that leave the constants fixed for all $1\leq p<\infty$. These representation spaces are either all separable, or all non-separable. If $G$ is a topological group, then these representations are either all strongly/weakly continuous, or all strongly/weakly discontinuous.
\end{corollary}

We will now proceed to show that, if $G$ is a locally compact Hausdorff group, there is a model that gives an addional `explanation' of Corollary~\ref{cor:family}, in addition to the observation in Remark~\ref{rem:lamperti}. The main ideas leading to the pertinent Theorem~\ref{thm:transfer} are those employed in the proof of \cite[Theorem 15.27]{EiFaHaNa15}, where the group is compact and $p=1$, but with a few technical modifications, so that they lead to a stronger result that is valid for non-compact groups (a rather modest achievement) and for all $1\leq p<\infty$ simultaneously. The basic tool is an application of the commutative Gelfand-Naimark theorem, and for this we need some preparations.

The proof of the following result is a technically strengthened variation on part of the  proof of \cite[Theorem~15.27]{EiFaHaNa15}. For this, invariant integration over the group is needed, and this is the reason that the requirement that $G$ be a locally compact Hausdorff group becomes part of the hypotheses.

\begin{lemma}\label{lem:dense_subalgebra_one_space}
Let $(\ts,\mu)$ be a probability space, let $1\leq p<\infty$, and let $\rep$ be a strongly continuous representation of a locally compact Hausdorff group $G$ on $\Ell^p(\ts,\mu)$ as isometric lattice automorphisms that leave the constants fixed. Then there exists a $G$-invariant closed subalgebra $A_p$ of $(\Ell^\infty(\ts,\mu), {\norm{\,\cdot\,}}_\infty)$ that contains $\ind_\ts$, is dense in $\Ell^p(\ts,\mu)$, and is such that the restricted representation of $G$ on $(A_p,{\norm{\,\cdot\,}}_\infty)$ is strongly continuous. If $\mu$ is separable, and $G$ is $\sigma$-compact, then $A_p$ can be taken to be a separable subalgebra of $(\Ell^\infty(\ts,\mu), {\norm{\,\cdot\,}}_\infty)$.
\end{lemma}

\begin{proof}
If $f\in\Ell^p(\ts,\mu)$, and $\phi\in \Cc(G)$, then, since the integrand is continuous and compactly supported, the $\Ell^p(\ts,\mu)$-valued Bochner integral  $\rep(\phi)f=\int_\ts \phi(g)\rep(g)f\,\ud\mu_G(g)$ exists; here $\mu_G$ is a left-invariant Haar measure on $G$. If $f\in\Ell^\infty(\ts,\mu$), then $\rep(\phi)$ is, in fact, an element of $\Ell^\infty(\ts,\mu)$. To see this, choose, for $n=1,2,\ldots$, a disjoint partition $\supp\, \phi=\bigcup_{i=1}^{N_n} E_i$ of the compact set $\supp\, \phi$ into measurable subsets $E_i$, and $g_i\in E_i$, such that ${\norm{\phi(g)\rep(g)f-\phi(g_i)\rep(g_i)f}}_p\leq 1/n$ and $|\phi(g)-\phi(g_i)|\leq 1/n$ for all $g\in E_i$. It is easy to see that $\sum_{i=1}^{N_n}\mu(E_i)\phi(g_i)\rep(g_i)f\to\rep(\phi)f$ in $\Ell^p(\ts, \mu)$ as $n\to\infty$. Passing to a subsequence, we may assume that this convergence is pointwise almost everywhere. On the other hand, we know that $G$ acts as isometries on $(\Ell^\infty(\ts,\mu),{\norm{\,\cdot\,}}_\infty)$, so that ${\norm{\sum_{i=1}^{N_n}\mu(E_i)\phi(g_i)\rep(g_i)f}}_\infty\leq\mu(\supp\,\phi){\norm{\phi}_\infty\norm{f}}_\infty$ for all $n$. We conclude that  $\rep(\phi)f$ is an element of $\Ell^\infty(\ts,\mu)$, as claimed. Moreover, since ${\norm{\sum_{i=1}^{N_n}\mu(E_i)\phi(g_i)\rep(g_i)f}}_\infty\leq\sum_{i=1}^{N_n}\leq |\phi(g_i)|\mu(E_i){\norm{\phi}}_\infty$, we can let $n\to\infty$ and conclude that ${\norm{\rep(\phi)f}}_\infty \leq {\norm{\phi}}_1{\norm{\phi}}_\infty {\norm{f}}_\infty$.
It follows easily from the latter inequality and the strong continuity of the left regular representation of $G$ on $\Ell^1(G)$ that the map $g\mapsto\rep(g)\rep(\phi)f$ from $G$ into $(\Ell^\infty(\ts,\mu),{\norm{\,\cdot\,}}_\infty)$ is continuous.

After these preparations, we let
\[
A^\prime_p=\{f\in\Ell^\infty(\ts,\mu) : g\mapsto\rep(g)f\textup{ is continous from }G\textup{ into }(\Ell^\infty(\ts,\mu),{\norm{\,\cdot\,}}_\infty)\}.
\]
Using that $G$ acts as isometries on $(\Ell^\infty(\ts,\mu),{\norm{\,\cdot\,}}_\infty)$, one sees that $A^\prime_p$ is a closed $G$-invariant subalgebra of $(\Ell^\infty(\ts,\mu),{\norm{\,\cdot\,}}_\infty)$ that contains $\ind_\ts$. It follows from our preparations that $A^\prime_p$ is dense in $\Ell^p(\ts,\mu)$.

For the general case, one can take $A_p=A^\prime_p$. If $\mu$ is separable, and $G$ is $\sigma$-compact, we select a countable subset $S$ of $A^\prime_p$ containing $\ind_\ts$ that is dense in $\Ell^p(\ts,\mu)$. If $f\in S$, then $\rep(G)f$ is a $\sigma$-compact, and hence a separable, subset of $(A^\prime_p,{\norm{\,\cdot\,}}_\infty)$. Therefore there exists a countable subset $G_f$ of $G$, containing the identity element, such that  $\rep(G)f\subseteq{\overline{\{\rep(g)f : g\in G_f\}}}^{{\norm{\,\cdot\,}}_\infty}\subseteq A^\prime_p$. One can now take $A_p$ to be the closed subalgebra of  $(\Ell^\infty(\ts,\mu), {\norm{\,\cdot\,}}_\infty)$ that is generated by the $\rep(g)f$ for $f\in S$ and $g\in G_f$.
\end{proof}

\begin{remark} It is worth noting that every separable locally compact Hausdorff group $G$ is $\sigma$-compact. Indeed, there exists an open neighbourhood of $V$ of $e$ in $G$ that is $\sigma$-compact (according to \cite[Proposition~2.4]{Folland95}, $V$ can even be taken to be an open closed subgroup), and if $S\subset G$ is a countable dense subset, then $G=\bigcup_{s\in S} s(V\cap V^{-1})$ is $\sigma$-compact.
\end{remark}

The following result has no counterpart in \cite{EiFaHaNa15}. It is needed when one wants to transfer the `whole' picture in Theorem~\ref{thm:transfer}, i.e.\ for all $1\leq p<\infty$ simultaneously.

\begin{proposition}\label{prop:dense_subalgebra_all_spaces}
Let $(\ts,\mu)$ be a probability space, let $1\leq p_0<\infty$, and let $\rep$ be a strongly continuous representation of a locally compact Hausdorff group $G$ on $\Ell^{p_0}(\ts,\mu)$ as isometric lattice automorphisms that leave the constants fixed, so that $G$ acts naturally in a similar fashion on $\Ell^p(\ts,\mu)$ for all $1\leq p<\infty$. Then there exists a $G$-invariant closed subalgebra $A$ of $(\Ell^\infty(\ts,\mu), {\norm{\,\cdot\,}}_\infty)$ that contains $\ind_\ts$, is dense in $\Ell^p(\ts,\mu)$ for all $1\leq p<\infty$, and is such that the restricted representation of $G$ on $(A,{\norm{\,\cdot\,}}_\infty)$ is strongly continuous. If $\mu$ is separable, and $G$ is $\sigma$-compact, then $A$ can be taken to be a separable subalgebra of $(\Ell^\infty(\ts,\mu), {\norm{\,\cdot\,}}_\infty)$.
\end{proposition}

\begin{proof}
For $n=1,2,\ldots$, choose an algebra $A_n$ as in Lemma~\ref{lem:dense_subalgebra_one_space} that is dense in $\Ell^n(\ts,\mu)$, and let $A$ be the closed subalgebra of $(\Ell^\infty(\ts,\mu), {\norm{\,\cdot\,}}_\infty)$  that is generated by the $A_n$.
\end{proof}

The following `transfer theorem' is a stronger version of \cite[Theorem~15.27]{EiFaHaNa15}. We include the short proof for the convenience of the reader, but hasten to add that it is a modest variation on that of \cite[Theorem~15.27]{EiFaHaNa15}, where only $p=1$ is considered and where the group is compact.

\begin{theorem}\label{thm:transfer}
Let $(\ts,\mu)$ be a probability space, let $1\leq p_0<\infty$, and let $\rep^{p_0}$ be a strongly continuous representation of a locally compact Hausdorff group $G$ on $\Ell^{p_0}(\ts,\mu)$ as isometric lattice automorphisms that leave the constants fixed, so that $G$ acts naturally in a similar fashion on $\Ell^p(\ts,\mu)$ for all $1\leq p<\infty$.

Then there exist
\begin{enumerate}
\item\label{thm:transfer_1}  a topological dynamical system $(G,K)$, where $K$ is a compact Hausdorff space;
\item\label{thm:transfer_2}  a $G$-invariant regular Borel probability measure $\widetilde\mu$ on $K$ with $\supp\, \widetilde\mu=K$;
\item\label{thm:transfer_3}  a family $\{\Phi_p\}_{1\leq p<\infty}$ of isometric lattice isomorphisms $\Phi_p:\Ell^p(\ts,\mu)\to\Ell^p(K,\widetilde\mu)$ that
\begin{enumerate}
\item\label{thm:transfer_3a} send $\ind_\ts$ to $\ind_K$;
\item\label{thm:transfer_3b} are compatible with the inclusions between $\Ell^p$-spaces;
\item\label{thm:transfer_3c} intertwine the strongly continuous representations of $G$ on the spaces $\Ell^p(\ts,\mu)$ with the canonical strongly continuous representations of $G$ on the spaces $\Ell^p(K,\widetilde\mu)$.
\end{enumerate}
\end{enumerate}
If $\mu$ is separable, and $G$ is $\sigma$-compact, then $K$ can be taken to be metrizable.
\end{theorem}

\begin{proof}
Choose an algebra $A$ as in Proposition~\ref{prop:dense_subalgebra_all_spaces}. By the commutative Gelfand-Naimark theorem, there exist a compact Hausdorff space $K$ and a unital isometric algebra isomorphism $\Phi:(A,{\norm{\,\cdot\,}}_\infty)\to (\textup{C}(K), {\norm{\,\cdot\,}}_\infty)$. If $A$ is separable, then $K$ is metrizable. We know from \cite[Theorem~7.23.(iv)-(vi)]{EiFaHaNa15} that $\Phi$ is a lattice isomorphism, that $|f|^p\in A$ for all $f\in A$ and $1\leq p<\infty$, and that
\begin{equation}\label{eq:pth_power}
\Phi({|f|}^p)={|\Phi(f)|}^p\quad (f\in A, \,1\leq p<\infty).
\end{equation}
We transfer the strongly continuous action of $G$ on $(A,{\norm{\,\cdot\,}}_\infty)$ to $(\textup{C}(K), {\norm{\,\cdot\,}}_\infty)$ via $\Phi$. As is well known, this transferred action necessarily originates from a topological dynamical system $(G,K)$.

The Riesz representation theorem furnishes a regular Borel probability measure $\widetilde\mu$ on $K$, easily seen to be of full support, such that
\begin{equation}\label{eq:riesz}
\int_K \Phi(f) \,\ud\widetilde\mu = \int_\ts f \,\ud\mu\quad(f\in A).
\end{equation}

Since $\Phi$ intertwines the $G$-actions on $\textup{C}(K)$ and $A$ by construction, it is immediate from \eqref{eq:riesz} and part~\ref{lem:markov_typering_1} of Lemma~\ref{lem:markov_typering} that $\widetilde\mu$ is $G$-invariant. Furthermore, combination of \eqref{eq:pth_power} and \eqref{eq:riesz} shows that
\[
\int_K {|\Phi(f)|}^p\,\ud\widetilde\mu=\int_K \Phi({|f|}^p)\,\ud\widetilde\mu = \int_\ts {|f|}^p\,\ud\mu\quad (f\in A,\,1\leq p<\infty).
\]
Since, for all $1\leq p<\infty$, $A$ is dense in $\Ell^p(\ts,\mu)$, and $\textup{C}(K)$ is dense in $\Ell^p(K,\mu)$, by extension we obtain a family of isometries $\Phi_p:\Ell^p(\ts,\mu)\to\Ell^p(K, \widetilde\mu)$ ($1\leq p<\infty$). Since $\Phi$ is a lattice isomorphism, so are the $\Phi_p$.  The statements in parts \ref{thm:transfer_3b} and \ref{thm:transfer_3c} are routinely verified.
\end{proof}

It is now clear that Theorems~\ref{thm:transfer} can still be used to disintegrate representations even when there is no initial action on the underlying point set, since\textemdash under mild conditions\textemdash the latter is furnished by Theorem~\ref{thm:transfer}. The result is the following, which should be compared with the general unitary disintegration in \cite[Theorem~18.7.6]{Dixmier77}. Note the separability assumption on the probability space, needed to ensure that the compact Hausdorff space from Theorem~\ref{thm:transfer} is Polish.

\begin{theorem}\label{thm:disintegrating_Markov_actions}
Let $G$ be a locally compact Polish group, let $1\leq p_0<\infty$, and let $(X,\mu)$ be a separable probability space. Let $\rep^{p_0}:G\to\La(\Ell^{p_0}(\ts,\mu))$ be a strongly continuous representation of $G$ as isometric lattice automorphisms that leave the constants fixed. Then, for all $1\leq p<\infty$, there exists a representation $\rep^{p}$ of $G$ on $\Ell^p(\ts,\mu)$ with the same properties, that is obtained from $\rho^{p_0}$ via restriction to, and extension from, $\Ell^\infty(\ts,\mu)$. Furthermore, there exist a Borel probability space $ (\Omega,\nu)$ and a vector space $\vs$ such that, for all $1\leq p<\infty$, there exist
\begin{enumerate}
\item a measurable family $\{\bs_{\w}^p\}_{\w\in\Omega}$ of Banach lattices over $(\Omega,\nu,\vs)$;
\item a family of strongly continuous and order indecomposable representations $\rep_{\w}^p:G\to\La(\bs_{\w}^p)$ \textup{(}$\w\in\Omega$\textup{)} of $G$ as isometric lattice isomorphisms of $\bs_{\w}^p$;
\item an isometric lattice isomorphism $\isom^p:\Ell^{p}(\ts,\mu)\to \left(\int_{\Omega}^{\oplus}\bs_{\w}^p\,\ud\nu(\w)\right)_\ks$ such that the following diagram commutes for all $g\in G$:
\begin{align*}
  \xymatrixcolsep{3.5pc}\xymatrix{
 \Ell^{p}(\ts,\mu) \ar[r]^{\rep^p(g)} \ar[d]_{\isom^p} &   \Ell^{p}(\ts,\mu) \ar[d]_{\isom^p}\\ \left(\int_{\Omega}^{\oplus}\bs_{\w}^p\,\ud\nu(\w)\right)_\ks\quad \ar[r]^{\left(\int_{\Omega}^{\oplus}\!\rep_{\w}^p(g)\,\ud\nu(\w)\right)_\ks}& \quad \left(\int_{\Omega}^{\oplus}\bs_{\w}^p\,\ud\nu(\w)\right)_\ks
 }
\end{align*}
\end{enumerate}
\end{theorem}

Inspection of the proofs shows that there is some more information available. $V$ can be taken to be the vector lattice of all simple functions on the compact metric space $K$ that is furnished by Theorem~\ref{thm:transfer}, $\Omega$ is then the set of all ergodic Borel probability measures on $K$, and $\nu$ is then the push-forward of the measure $\widetilde\mu$ on $K$ in Theorem~\ref{thm:transfer} to the set of ergodic Borel probability measures $\Omega$, using a decomposition map for $\widetilde\mu$ as in Section~\ref{sec:disintegrating_space_actions}. The family $\{\bs_{\w}^p\}_{\w\in\Omega}$ of Banach lattices is then the family $\{\Ell^{p}(K,\w)\}_{\w\in\Omega}$ of $\Ell^{p}$-spaces corresponding to the ergodic Borel probability measures on $K$, and the representations $\rep_{\w}^p$ are then the canonical representations of $G$ on these spaces.


\section{Perspective}\label{sec:perspective}

In Section~\ref{sec:introduction}, we put forward the task of disintegrating strongly continuous representations of a locally compact group as isometric lattice automorphisms of Banach lattices into similar representations that are order indecomposable. This would be the analogue of what is known to be possible for strongly continuous unitary representations of separable groups on separable Hilbert spaces. The $\Ell^p$-spaces for finite $p$ are arguably the prime examples of Banach lattices that can serve as representation spaces, and in that case our goal was achieved in Theorem~\ref{thm:disintegrating_Markov_actions} for a certain class of such representations. As explained in Section~\ref{sec:introduction}, this class already includes e.g.\ all natural representations on $\Ell^p$-spaces corresponding to topological actions of Lie groups on compact manifolds with an invariant Borel probability measure. Consequently, we now do not only know that the ensuing natural unitary representation of the group on the pertinent (complex) $\Ell^2$-space is a direct integral of irreducible (i.e.\ indecomposable) unitary representations, but also that the natural representations of the group as isometric lattice automorphisms of the pertinent real $\Ell^p$-spaces for finite $p$ are direct integrals of similar representations that are order indecomposable. 

Still, it is clear that Theorem~\ref{thm:disintegrating_Markov_actions} is only a first step in the study of the disintegration of general strongly continuous group representations as isometric lattice automorphisms of $\Ell^p$-spaces. At a conceptual level, the main insight seems to be that this is, in fact, possible for the representation in the present paper, and that (a modification of) the direct integral theory in \cite{HaLeRa91} provides the language to formalize such a disintegration. This is not so clear at the outset.

It is hoped that further steps can be taken. One possible development, still for a probability measure $\mu$ and a strongly continuous representation as isometric lattice automorphisms that leave the constants fixed, would be to attempt to relax the conditions in Theorem~\ref{thm:disintegrating_Markov_actions} that $G$ be Polish and/or that the probability space be separable. As is indicated in Remark~\ref{rem:lamperti}, if the constants are fixed, then one is `actually' looking at a measure preserving action of $G$ on the measure algebra $\mathcal A_\mu$. It is conceivable that Maharam's work in \cite{Maharam50} can then be used to improve on the technical hypotheses in Theorem~\ref{thm:disintegrating_Markov_actions}, since the main basic results (Theorems~1,~2a, and~2b) in \cite{Maharam50} do not involve any topology. They can be applied in the context of any measure preserving abstract group action on a measure algebra, and yield a decomposition of $\mathcal A_\mu$ with respect to the sub-algebra of the fixed points of $G$ in $\mathcal A_\mu$. It is shown (see \cite[Theorems~6 and~7]{Maharam50}) that this can be used to yield an ergodic decompositon of the group action at the level of measure spaces if $G$ equals the integers or the real numbers (which are, incidentally, both still Polish), and it is mentioned (but not proved) that a similar theorem holds in more general cases. It is open to investigation whether such a decomposition at the level of measure spaces\textemdash once actually established for more general $G$\textemdash can be pushed still further to the $G$-action on the $\Ell^p$-spaces themselves, while at the same time incorporating the (modified) direct integral formalism of \cite{HaLeRa91}. There are definitely some measurability issues to be taken care of, and perhaps the assumptions on $G$ and $\mu$ in Theorem~\ref{thm:disintegrating_Markov_actions} are not only not too restrictive from a practical point of view, but also not so easy to avoid when needing to ensure measurability in the proofs. After all, for the disintegration of a strongly continuous unitary group representation both the group and the Hilbert space are also required to be separable. On a positive note, since our main sources for the ergodic decomposition in the present paper, Farrell's 1962 paper~\cite{Farrell62}, Varadarajan's 1963 paper \cite{Varadarajan63}, and Zakrzewski's 2002 overview \cite{Zakrzewski02}, make no use Maharam's considerably earlier 1950 paper \cite{Maharam50} (in fact, they do not refer to her work at all), but concentrate on Borel spaces and group actions thereon, such a general approach based on \cite{Maharam50} would, to the knowledge of the authors, certainly provide a new angle on the matter.

Another possible development is the bold leap to consider the most general case of strongly continuous representations as isometric lattice automorphisms of $\Ell^p$-spaces\textemdash for possibly infinite measure $\mu$\textemdash that do not necessarily arise from an underlying measure preserving action. By Lamperti's theorem \cite[Theorem~3.2.5]{Fleming-Jamison03}, such operators are\textemdash this is true for $\sigma$-finite measures\textemdash always a composition of a multiplication operator and an operator that arises from an (not necessarily measure preserving) action on $\mathcal A_\mu$; see also Remark~\ref{rem:lamperti} for $p=2$. With this factorisation available, as a next step one could e.g.\ try to adapt the approach via Borel spaces as in \cite{Farrell62}, \cite{Varadarajan63}, or \cite{Zakrzewski02}, or else attempt a route via measure algebras by generalizing the material in \cite{Maharam50}. 

\subsubsection{Acknowledgements}
We thank Markus Haase for pointing out various results in~\cite{EiFaHaNa15}, specifically~\cite[Theorem 15.27]{EiFaHaNa15}, and for discussions on the proofs of the latter result and our Lemma~\ref{lem:dense_subalgebra_one_space}. The authors are indebted to the anonymous referee for pointing out the possible potential of Maharam's work in \cite{Maharam50} for the current line of research.

\bibliographystyle{plain}
\bibliography{Bibliografie}

\end{document}